\documentclass[11 pt]{amsart}
\usepackage{amsmath}
\usepackage{latexsym,mathrsfs}
\usepackage{mathdots}
\usepackage{color}
\usepackage{enumitem} 
\usepackage{bbm}
\usepackage{tikz}
\usepackage{tikz-cd}
\usepackage[T1]{fontenc}
 \usepackage[utf8]{inputenc}
 \usepackage{blkarray} 
 \usepackage{mathtools} 
 \usepackage{longtable}
 \usepackage{multicol} 
  \usepackage{fdsymbol} 
 \usepackage{hyperref}
\usepackage[capitalise]{cleveref}

\usetikzlibrary{positioning} 
\usetikzlibrary{decorations.pathreplacing}
\usetikzlibrary{arrows, decorations.markings} 
\pgfarrowsdeclarecombine{twotriang}{twotriang}{stealth}{stealth}%
{stealth}{stealth}
\tikzset{arrow/.style={-stealth}}
\tikzset{arrowshorter/.style={-stealth, shorten <=2pt, shorten >=2pt}}
\tikzset{arrowmuchshorter/.style={-stealth, shorten <=7pt, shorten >=6pt}}
\tikzset{mono/.style={>-stealth}} 
\tikzset{epi/.style={-twotriang}} 
\tikzset{twoarrowlonger/.style={double,double distance=1.5pt,
shorten <=5pt,shorten >=6pt,
decoration={markings,mark=at position -4pt with {\arrow[scale=1.75]{>}}},
preaction={decorate}}} 

\tikzset{twoarrow/.style={double,double distance=1.5pt,
shorten <=6pt,shorten >=7pt, 
decoration={markings,mark=at position -4pt
with {\arrow[scale=1.75]{>}}},
preaction={decorate} 
}
}
\tikzset{%
    symbol/.style={%
        draw=none,
        every to/.append style={%
            edge node={node [sloped, allow upside down, auto=false]{$#1$}}}
    }
}

\tikzset{mapstikz/.style={-stealth, 
decoration={markings,mark=at position 0pt with {\arrow[scale=0.5]{|}}}, preaction={decorate}}}
\usetikzlibrary{backgrounds,shapes,arrows,calc,patterns} 

\pgfdeclarelayer{background}
\pgfsetlayers{background,main}

\theoremstyle{plain}   
\newtheorem{thm}{Theorem}[section] 
\makeatletter\let\c@thm\c@thm\makeatother

\makeatletter\let\c@cor\c@thm\makeatother
\newtheorem{lem}{Lemma}[section]
\makeatletter\let\c@lem\c@thm\makeatother
\newtheorem{prop}{Proposition}[section]
\makeatletter\let\c@prop\c@thm\makeatother

\makeatletter\let\c@claim\c@thm\makeatother

\makeatletter\let\c@conjecture\c@thm\makeatother

\newtheorem*{unnumberedtheorem}{Theorem}

\theoremstyle{definition}

\newtheorem{defn}{Definition}[section]
\makeatletter\let\c@defn\c@thm\makeatother

\makeatletter\let\c@const\c@thm\makeatother
\newtheorem{notn}{Notation}[section]
\makeatletter\let\c@notn\c@thm\makeatother

\makeatletter\let\c@convention\c@thm\makeatother

\theoremstyle{remark}

\newtheorem{rmk}{Remark}[section]
\makeatletter\let\c@rmk\c@thm\makeatother

\makeatletter\let\c@ex\c@thm\makeatother

\makeatletter\let\c@observation\c@thm\makeatother

\makeatletter\let\c@warning\c@thm\makeatother

\makeatletter\let\c@digression\c@thm\makeatother

\makeatletter\let\c@answ\c@thm\makeatother

\makeatletter\let\c@answ\c@thm\makeatother

\makeatletter\let\c@aside\c@thm\makeatother

\makeatletter
\let\c@equation\c@thm
\numberwithin{equation}{section}
\makeatother

\newcommand{\newrefformat}[2]{}

\crefname{lem}{Lemma}{Lemmas}
\crefname{thm}{Theorem}{Theorems}
\crefname{defn}{Definition}{Definitions}
\crefname{notn}{Notation}{Notations}
\crefname{const}{Construction}{Constructions}
\crefname{prop}{Proposition}{Propositions}
\crefname{rmk}{Remark}{Remarks}
\crefname{cor}{Corollary}{Corollaries}
\crefname{equation}{Display}{Displays}
\crefname{ex}{Example}{Examples}
\crefname{thmalph}{Theorem}{Theorems}
\crefname{answ}{Answer}{Answers}
\crefname{crzcond}{Crazy Condition}{Crazy Conditions}

\newcommand{\cA}{\mathcal{A}}

\newcommand{\cC}{\mathcal{C}}
\newcommand{\cD}{\mathcal{D}}
\newcommand{\cE}{\mathcal{E}}

\newcommand{\cG}{\mathcal{G}}

\newcommand{\cI}{\mathcal{I}}

\newcommand{\cM}{\mathcal{M}}

\newcommand{\cO}{\mathcal{O}}
\newcommand{\cP}{\mathcal{P}}

\newcommand{\cS}{\mathcal{S}}

\newcommand{\glboundary}{\partial^{\mathrm{gl}}}
\newcommand{\cat}{\cC\!\mathit{at}}
\newcommand{\set}{\cS\!\mathit{et}}
\newcommand{\sset}{\mathit{s}\set}

 \renewcommand*\colon{%
   \nobreak
   \mskip2mu
   \mathpunct{}%
   \nonscript
   \mkern-\thinmuskip
   {:}%
   \mskip6mu plus1mu
   \relax
 }   

 \newcommand*\noloc{%
   \nobreak
   \mskip6mu plus1mu
   \mathpunct{}%
   \nonscript
   \mkern-\thinmuskip
   {:}%
   \mskip2mu
   \relax
}

\DeclareMathOperator{\colim}{colim}
 \DeclareMathOperator{\id}{id}

\author{Viktoriya Ozornova}
\address{Max Planck Institute for Mathematics, Bonn, Germany}
\email{viktoriya.ozornova@mpim-bonn.mpg.de}

\author{Martina Rovelli}
\address{
University of Massachusetts Amherst, Amherst, USA}
\email{mrovelli@umass.edu} 
\address{
University of Ottawa, 
Ottawa,
Canada
}
\email{mrovelli@uottawa.ca}

\author{Tashi Walde}
\address{Universität Regensburg, Regensburg, Germany}
\email{tashi.walde@ur.de}

\DeclareMathOperator{\Cone}{Cone}
\usepackage{boxedminipage}
\tikzcdset{arrow style=tikz,
           diagrams={>=stealth}
           }

\usepackage[margin=1.4in]{geometry}

\title{About the contractibility of the walking coinductive equivalence}

\begin{document}

\maketitle

\begin{abstract}
We study the marked simplicial set obtained as the Roberts--Street nerve of the walking coinductive equivalence. We show that it is a non-contractible saturated complicial set for which all of its finite truncations are contractible. When regarding saturated complicial sets as a model for right $(\infty,\infty)$-categories, it represents a concrete and explicit example of a right $(\infty,\infty)$-category that is itself non-contractible, but whose reflection to a left $(\infty,\infty)$-category is contractible.
\end{abstract}

\addtocontents{toc}{\protect\setcounter{tocdepth}{1}}
\tableofcontents

\section*{Introduction}
The inclusion
\[
\mathscr Cat_{(\infty,n)}\hookrightarrow\mathscr Cat_{(\infty,n+1)}
\]
of the $\infty$-category of $(\infty,n)$-categories into that of $(\infty,n+1)$-categories admits both a left and a right adjoint, giving rise to two towers. It was established in \cite{GH-hyper,ORW1} that the limit of the tower of left adjoints, denoted $\mathscr Cat_{(\infty,\infty)}^{L}$, is a localization of the limit of the tower of right adjoints, denoted $\mathscr Cat_{(\infty,\infty)}^{R}$, through an adjunction
\begin{equation} \label{localization} \tag{$\star$} L\colon \mathscr Cat_{(\infty,\infty)}^{R}\rightleftarrows \mathscr Cat_{(\infty,\infty)}^{L}\noloc R. \end{equation}
Despite the existence of this localization, a precise understanding of the information it discards remains a central challenge.

Verity introduced \emph{complicial sets} \cite{VerityComplicialI}, and later \emph{saturated complicial sets} \cite{EmilyNotes,or,RiehlVerityBook,VeritySlides}; the theory has since been further developed in \cite{Loubaton4}. These are simplicial sets endowed with a class of marked simplices satisfying axioms that ensure the marked simplices represent weakly invertible cells when the marked simplicial set is regarded as a higher category. In particular, saturated complicial sets in which every simplex of dimension greater than $n$ is marked model $(\infty,n)$-categories.

The homotopy theory of saturated complicial sets for which all simplices are marked in dimension higher than $n$, encoded by the model category $m\sset_{(\infty,n)}$, is a model for $(\infty,n)$-categories $\mathscr Cat_{(\infty,n)}$, and the homotopy theory of saturated complicial sets, encoded by the model category $m\sset_{\mathrm{cmp,sat}}$, was intended as a model for the right limit of the tower of $(\infty,n)$-categories $\mathscr Cat_{(\infty,\infty)}^{R}$.

Assuming this identification, the reflector $L$ in the localization \eqref{localization} is computed levelwise by applying the thinification functors $\mathrm{th}_n$. More precisely, it sends a saturated complicial set $X$ to the compatible family
$\{\mathrm{th}_nX\}_{n\geq0}$,
where $\mathrm{th}_nX$ is obtained from $X$ by freely marking every simplex of dimension greater than $n$.

In past work \cite{ORsurvey,HLOR}, an explicit $\omega$-category $\cE$ which is gaunt, contractible, and non-trivial was constructed. In this paper we study its Roberts--Street nerve $N^{RS}\cE$; that is, its nerve $N\cE$ when endowed with the marking of simplices for which the top dimensional simplex is an identity in $\cE$. We establish as \cref{CompleteTheorem} (combining \cref{NEfibrant,NEnotContractible,LnEcontractible}) the following properties of $N^{\mathrm{RS}}\cE$:

\begin{unnumberedtheorem}
Let $\cE$ denote the walking coinductive equivalence from \cite{ORsurvey,HLOR}, let $N\cE$ be its simplicial nerve, and let $N^{\mathrm{RS}}\cE$ be its Roberts--Street nerve. Then:
\begin{enumerate}
  \setcounter{enumi}{-1}
    \item$N^{\mathrm{RS}}\cE$ is a saturated complicial set;
    \item $N^{\mathrm{RS}}\cE$ is not contractible in the homotopy theory of saturated complicial sets;
    \item for every $n\geq0$, the thinification $\mathrm{th}_nN^{\mathrm{RS}}\cE$ is contractible in the homotopy theory of saturated complicial sets.
\end{enumerate}
\end{unnumberedtheorem}

In the intended interpretation of saturated complicial sets as a model for $\mathscr Cat_{(\infty,\infty)}^{R}$,
the object $N^{\mathrm{RS}}\cE$ provides an explicit example in the model of complicial sets of an object that is non-contractible in $\mathscr Cat_{(\infty,\infty)}^{R}$, but whose image $L(N^{\mathrm{RS}}\cE)$ in $\mathscr Cat_{(\infty,\infty)}^{L}$ is contractible. In particular, the example also produces a map $N^{RS}\cE\to\Delta[0]$ that is not an equivalence in $\mathscr Cat_{(\infty,\infty)}^{R}$, but whose image $L(N^{RS}\cE\to\Delta[0])$ in $\mathscr Cat_{(\infty,\infty)}^{L}$ is an equivalence. Another example of a map that is not an equivalence in $\mathscr Cat_{(\infty,\infty)}^{R}$ but whose image in $\mathscr Cat_{(\infty,\infty)}^{L}$ is one was provided in \cite[Construction 4.33]{HL}.
In a similar vein, one could possibly exploit the key constructions and ideas from \cite{Cheng} to show that a to-be-designed saturated complicial set of higher bordisms is not equivalent to an $(\infty,0)$-category, but becomes so after reflecting. The example we provide remains in a sense smaller and more minimal than these alternative approaches.

\subsubsection*{Acknowledgments}
MR is grateful for support from the NSF grant DMS-2203915 and NSERC grant RGPIN-2026-05664.
We are grateful to Dimitri Ara (for discussions around \cref{AcyclicCofOmegaCat0}), Amar Hadzihasanovic (for discussions around
\cref{IteratedSuspensionMarked,SurvivingNonDegenerate}), F\'elix Loubaton (for discussions around \cref{CvsCone}) and Lennart Meier for insightful conversations on this topic.

\subsubsection*{AI use disclosure}
This article contains no AI-generated content.
ChatGPT~5.6 Sol was used with read-only access
for light proofreading in the final stage of manuscript preparation.


\section{About $\omega$-categories}
\label{OmegaCategories}

\subsection{Background on $\omega$-categories}

While we refer the reader to e.g.~\cite{StreetOrientedSimplexes}, \cite{AraMaltsiniotisJoin}, \cite[\textsection14.2]{PolygraphBook} for the detailed definition of an $\omega$-category, we briefly recall the main features here.

The data of a \emph{strict $\omega$-category} (henceforth just called \emph{$\omega$-category})
$\cD$ consists of a sequence of sets $\cD_k$ (for ${k \geq 0})$,
where $\cD_0$ is called the set of \emph{objects} of $\cD$ and $\cD_k$ (for $k>0$)
is the set of \emph{$k$-cells} or \emph{$k$-morphisms} or cells of \emph{dimension} $k$ of $\cD$, together with:
\begin{itemize}[leftmargin=*]
    \item \emph{source} and \emph{target} operators $s_k, t_k \colon \cD_m \to \cD_k$
    for all $m > k \geq 0$;
    \item \emph{identity} operators $\id_m \colon \cD_k \to \cD_{m}$ for all $m\geq k\ge0$;
    \item \emph{composition} operators $\circ_k \colon \cD_m \times_{\cD_k} \cD_m \to \cD_m$
    defined for all $m > k \geq 0$ and all pairs of $m$-cells $(g, f)$ such that $s_k(g) = t_k(f)$.
\end{itemize}
This data is subject to appropriate conditions of associativity, unitality and interchange.

An \emph{$\omega$-functor} $F \colon \cD \to \cD'$ between $\omega$-categories $\cD$ and $\cD'$
is a sequence of maps $F_k \colon \cD_k \to \cD'_k$
(for $k \geq 0$)
that preserves source, target, identity, and composition.
We denote by $\omega\cat$ (resp.~$n\cat$ for $n\geq0$) the category of $\omega$-categories and $\omega$-functors (resp.~$n$-categories and $n$-functors), as considered e.g.~in \cite[\textsection14.2]{PolygraphBook}, \cite{AraMaltsiniotisJoin}, \cite{StreetOrientedSimplexes}.
There is a full inclusion $n\cat\hookrightarrow\omega\cat$.

We collect here the specific notations for classical $\omega$-categories that will be used throughout the paper.

\begin{notn}
  \label{not:notable-omega-cats}
We consider the following $\omega$-categories.
\begin{itemize}[leftmargin=*]
\item We denote by $\cI$ the \emph{walking $1$-isomorphism}, and more generally for $k\geq0$ by $\Sigma^k\cI$ the \emph{walking $(k+1)$-isomorphism}.
\item For $k\geq0$, we denote by $\cC[k]$ the \emph{walking $k$-cell} and by $\partial\cC[k]$ its \emph{boundary}.
\item For $k\geq0$, we denote by $\cO[k]$ the \emph{$k$-th oriental} (see~e.g.~\cite{StreetOrientedSimplexes} for more details).
\end{itemize}
\end{notn}

We collect here the notations for some classical specific constructions on $\omega$-categories that will be used throughout the paper.

\begin{notn}
\label{TruncationsCat}
Let $\cD$ be an $\omega$-category.
\begin{itemize}[leftmargin=*]
    \item For $n\geq0$, the $n$-category $\mathrm{core}_n\cD$ is the $n$-th \emph{core} of $\cD$ from \cite[\textsection1.2]{AraMaltsiniotisJoin} (under the notation $\tau^b_{\leq n}\cD$). It is obtained from $\cD$ by forgetting all morphisms in dimensions higher than $n$. The $n$-th core defines a functor $\mathrm{core}_n\colon\omega\cat\to n\cat$. 
    \item For $n\geq0$, the $n$-category $\mathrm{tr}_n\cD$ is the $n$-th \emph{truncation} of $\cD$ from \cite[\textsection1.2]{AraMaltsiniotisJoin} (under the notation $\tau^i_{\leq n}\cD$). It is obtained from $\cD$ by identifying two $n$-morphisms whenever there exists a zig-zag of $(n+1)$-morphisms between them (and otherwise forgetting the morphisms in dimension higher than $n$ and leaving the ones in dimension $<n$ unchanged).
    The $n$-th truncation defines a functor $\mathrm{tr}_n\colon\omega\cat\to n\cat$.
\end{itemize}
\end{notn}

\begin{rmk}
\label{TruncationsAdjunctionCat}
As discussed in \cite[\textsection1.2]{AraMaltsiniotisJoin}, for $n\geq0$, the canonical inclusion $I_n\colon n\cat\hookrightarrow\omega\cat$, the truncation and core constructions from \cref{TruncationsCat} form adjunctions
\[
\mathrm{tr}_n\colon \omega\cat\rightleftarrows n\cat\noloc I_n\quad\text{ and }\quad I_n\colon n\cat\rightleftarrows\omega\cat\noloc\mathrm{core}_n.
\]
\end{rmk}

\begin{rmk}
\label{ChaosAdjunctionCat}
As noted in \cite[\textsection1.2]{AraMaltsiniotisJoin},
for $n\geq0$, the core construction from \cref{TruncationsCat} fits into an adjunction
\[
\mathrm{core}_n\colon \omega\cat\rightleftarrows n\cat\noloc\mathrm{ch}_n.
\]
Roughly speaking, the right adjoint $\mathrm{ch}_n\colon n\cat\to\omega\cat$ associates with
each $n$-category $\cD$ an $\omega$-category $\mathrm{ch}_n\cD$ obtained from $\cD$ by adding a unique $(n+1)$-morphism between any two parallel $n$-morphisms in $\cD$.
In particular, the core construction preserves colimits.
\end{rmk}
    
\begin{notn}
\label{SuspensionCat}
Let $\cD$ be an $\omega$-category.
\begin{itemize}[leftmargin=*]
    \item The $\omega$-category $\Cone\cD\coloneq\cD\star[0]$ is the \emph{join} of $\cD$ with the terminal $\omega$-category $[0]$
    constructed in \cite[\textsection6]{AraMaltsiniotisJoin}. Although the construction is involved, we will only use that $\Cone\cO[k]\cong\cO[k+1]$ for $k\geq0$ and that the join construction defines a cocontinuous functor $\mathrm{Cone}\colon \omega\cat\to{^{[0]/}\omega\cat}$.
    \item The $\omega$-category $\Sigma\cD$ is the \emph{two-point suspension} of $\cD$ from \cite[\textsection2.2]{ORquillen} (cf.~also \cite[\textsection B.6.5]{AraMaltsiniotisJoin} for a variant of the same construction). It has two objects $\bot,\top$, and homs given by
    \[(\Sigma\cD)(\bot,\bot)=(\Sigma\cD)(\top,\top)=\cC[0],\quad(\Sigma\cD)(\bot,\top)=\cD,\quad(\Sigma\cD)(\top,\bot)=\varnothing.\]
    Combining \cite[Corollaire B.5.6, Corollaire B.6.6]{AraMaltsiniotisJoin}, we obtain a natural isomorphism
    of $\omega$-categories
    \[\Sigma\cD\cong(\cD\star[0])/\cD.\]
    The suspension construction defines a cocontinuous functor $\Sigma\colon \omega\cat\to{^{\partial\cC[1]/}\omega\cat}$.
    \item Given objects $d,d'$ in $\cD$, the $\omega$-category $\cD(d,d')$ is the \emph{hom} $\omega$-category in $\cD$ from $d$ to $d'$ from e.g.~\cite[\textsection1.1]{HLOR}.  The set of $k$-cells of $\cD(d,d')$ is given by
    \[
    (\cD(d,d'))_k=\left({^{\partial\cC[1]/}\omega\cat}\right)((\Sigma\cC[k],\bot,\top),(\cD,d,d')).
    \]
    Roughly speaking, the set of $k$-cells of $\cD(d,d')$ is given by the $(k+1)$-cells of $\cD$ whose $0$-dimensional source and target are $d$ and $d'$, respectively.
    The hom construction defines a functor $(-)(-,-)\colon ^{\partial\cC[1]/}\omega\cat\to\omega\cat$.
\end{itemize}
\end{notn}

\begin{rmk}
\label{SuspensionAdjunctionCat}
As mentioned e.g.~in \cite[Proposition 1.1]{HLOR}, the suspension-hom pair forms an adjunction
\[
\Sigma\colon\omega\cat\rightleftarrows{^{\partial\cC[1]/}\omega\cat}\noloc(-)(-,-).
\]
\end{rmk}

We now consider iterated versions of suspension and hom:

\begin{prop}
\label{IteratedSuspensionAdjunctionCat}
Let $\ell\geq0$. The adjunction from \cref{{SuspensionAdjunctionCat}} induces an adjunction
\[\Sigma^{\ell}\colon\omega\cat
\rightleftarrows
{}^{\partial\cC[\ell]/}\omega\cat\noloc(-)^{(\ell)}(-),\]
where ${}^{\partial\cC[\ell]/}\omega\cat$ denotes the category of $\omega$-categories under $\partial\cC[\ell]$.
\end{prop}

\begin{proof}
This is an instance of \cref{IteratedAdjunction}, applied to
$\cC=\omega\cat$, $c_0=\varnothing$, $F=\Sigma$, and $G=(-)(-,-)$, using the suspension-hom adjunction of \cref{SuspensionAdjunctionCat}
and the identification $\partial\cC[\ell]\cong\Sigma^\ell\varnothing$.
\end{proof}

\begin{lem}
\label{IteratedAdjunction}
Let $F\colon\cC\to\cC$ be a functor, and $c_0$ be an initial object of $\cC$. If $F$ induces an adjunction
\begin{equation}
  \label{eq:original-ell=-1-adjunction}
  F\colon
  \cC \simeq {}^{c_0/}\cC
  \rightleftarrows
  {}^{F(c_0)/}\cC
  \noloc G
\end{equation}
then for all $\ell\geq0$ there is an adjunction
\[
F^\ell\colon\cC\simeq c_0/\cC
\rightleftarrows
{}^{F(c_0)/}\cC
\rightleftarrows
{}^{F^2(c_0)/}\cC
\rightleftarrows
\cdots
\rightleftarrows
{}^{F^\ell(c_0)/}\cC\noloc G^\ell.
\]
\end{lem}

\begin{proof}
We show this by induction on $\ell\geq0$. For $\ell=0$, using that $c_0$ is initial, this is just the canonical equivalence
\[
\cC\simeq{}^{c_0/}\cC,
\]
and for $\ell=1$ this is the given adjunction
\eqref{eq:original-ell=-1-adjunction}.
Now, by induction, assume that there is the adjunction
\begin{equation}
\label{eq:induction}
F^\ell\colon\cC\rightleftarrows{}^{F^\ell(c_0)/}\cC\noloc G^\ell.
\end{equation}
Furthermore, by taking the slice of the original adjunction
\eqref{eq:original-ell=-1-adjunction}
under the object $F^\ell(c_0)$,
we obtain an adjunction
\begin{equation}
  \label{eq:sliced-ell-adjunction}
  {}^{F^\ell(c_0)/}\cC
  \rightleftarrows
  {}^{F(F^\ell(c_0))/}\left ({}^{F(c_0)/}\cC\right).
\end{equation}
Using the identification $F(F^\ell(c_0))= F^{\ell+1}(c_0)$, 
we also obtain a canonical equivalence
\begin{equation}
  \label{eq:slice-equivalence-adju}{}^{F(F^\ell(c_0))/}\left({}^{F(c_0)/}\cC\right)
  \simeq
  {}^{F^{\ell+1}(c_0)/}\cC.
\end{equation}
Composing the three adjunctions
\eqref{eq:induction},
\eqref{eq:sliced-ell-adjunction} and
\eqref{eq:slice-equivalence-adju},
we obtain the adjunction
\[
F^{\ell+1}\colon\cC\rightleftarrows{}^{F^{\ell+1}(c_0)/}\cC\noloc G^{\ell+1},
\]
as desired.
\end{proof}

\subsection{Descriptions for cones of cells}

The goal of this subsection is to prove \cref{ConeCell3}, which gives a formula for cone of the $\ell$-cell $\cC[\ell]$ in terms
of the cone of the $(\ell+1)$-cell $\cC[\ell+1]$, for $\ell\geq0$. We will make use of this formula in the proof of \cref{CvsCone}.

\begin{rmk}
By direct inspection, one can see that there is an isomorphism of $\omega$-categories
\[
\Cone\cC[0]\cong\cC[1].
\]    
\end{rmk}

First, we use existing results to record how many cells need to be added to the cones of $\cC[\ell-1]$ to obtain the cone of $\cC[\ell]$.

\begin{lem}
\label{ConeCell1}
For all $\ell>0$, the cone $\Cone\cC[\ell]$ can be expressed as an iterated pushout of the form
\[
\begin{tikzcd}
  \partial\cC[\ell-1]
  \arrow[rr]
  \arrow[d,hook]
  \arrow[drr, phantom, "\pushout", very near end,
  yshift=0.2cm, xshift=0.1cm]
  &
  &
  \Cone\cC[\ell-1]\arrow[d]
  &
  &
  \\
  \cC[\ell-1]
  \arrow[r,hookrightarrow]
  &
  \partial\cC[\ell]
  \arrow[d,hook]
  \arrow[r]
  \arrow[dr, phantom, "\pushout",
  very near end, yshift=0.2cm, xshift=0.1cm]
  &
  \cA[\ell]\arrow[d]
  &
  &
  \\
  &
  \cC[\ell]\arrow[r]
  &
  \cA'[\ell]
  \arrow[d]
  &
  &
  \partial\cC[\ell]\arrow[ll]
  \arrow[d,hook]
  \arrow[dll, phantom, "\text{\reflectbox{$\pushout$}}",
  very near end, yshift=0.2cm, xshift=-0.1cm]
  \\
  &
  &
  \cA''[\ell]\arrow[d]
  &
  \partial\cC[\ell+1]
  \arrow[l]
  \arrow[d,hook]
  &
  \cC[\ell]
  \arrow[l,hookrightarrow]
  \arrow[dll, phantom, "\text{\reflectbox{$\pushout$}}",
  very near end, yshift=0.2cm, xshift=-0.1cm]
  \\
  &
  &
  \Cone\cC[\ell]
  &
  \cC[\ell+1]\arrow[l]
  &
\end{tikzcd}
\]
Moreover, the vertical composite
\[
\Cone\cC[\ell-1]\to\cA[\ell]\to\cA'[\ell]\to\cA''[\ell]\to\Cone\cC[\ell]
\]
is induced by the cotarget map $\cC[\ell-1]\to\cC[\ell]$.
\end{lem}

To provide intuition, let's unpack the construction for the case $\ell=1$.

\begin{rmk}
The iterated pushout from \cref{ConeCell1} for $\ell=1$ can be described as
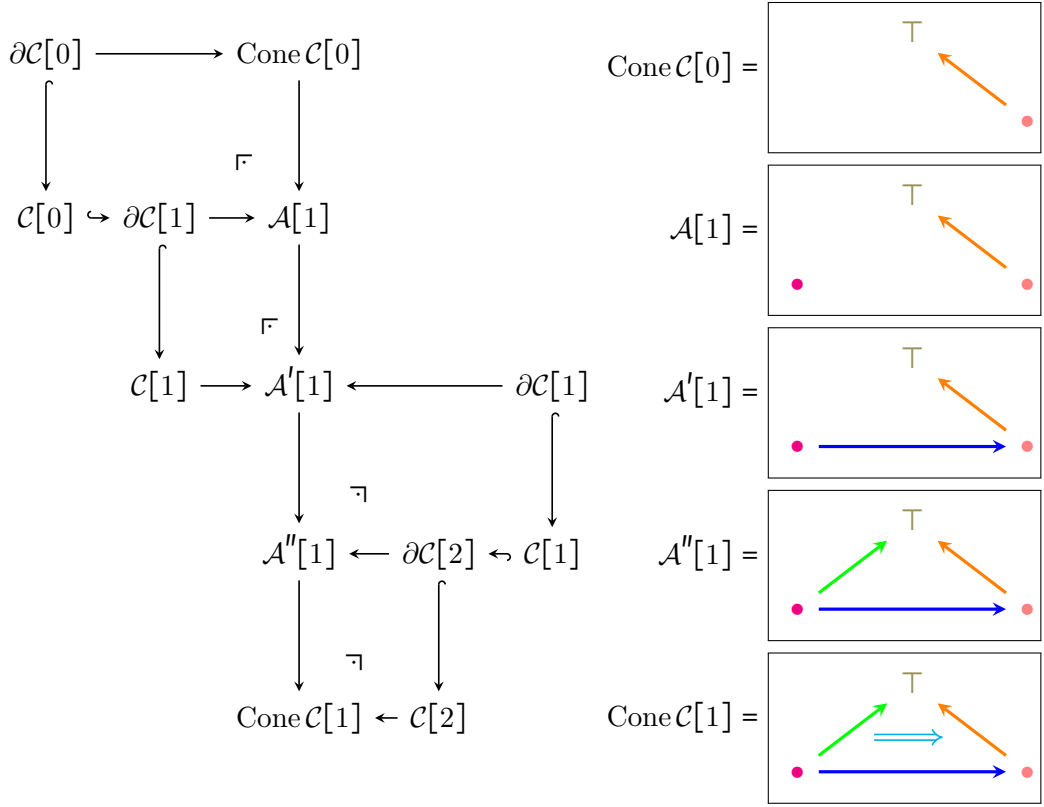
\begin{minipage}[t]{0.5\textwidth}
\begin{flushleft}
 \vspace{0.5cm}
\begin{tikzcd}[column sep=0.2cm, row sep=1.48cm]
  \partial\cC[0]
  \arrow[rr]
  \arrow[d,hook]
  \arrow[drr, phantom, "\pushout", very near end,
  yshift=0.2cm, xshift=0.1cm]
  &
  &
  \Cone\cC[0]\arrow[d]
  &
  &
  \\
  \cC[0]
  \arrow[r,hookrightarrow]
  &
  \partial\cC[1]
  \arrow[d,hook]
  \arrow[r]
  \arrow[dr, phantom, "\pushout",
  very near end, yshift=0.2cm, xshift=0.1cm]
  &
  \cA[1]\arrow[d]
  &
  &
  \\
  &
  \cC[1]\arrow[r]
  &
  \cA'[1]
  \arrow[d]
  &
  &
  \partial\cC[1]\arrow[ll]
  \arrow[d,hook]
  \arrow[dll, phantom, "\text{\reflectbox{$\pushout$}}",
  very near end, yshift=0.2cm, xshift=-0.1cm]
  \\
  &
  &
  \cA''[1]\arrow[d]
  &
  \partial\cC[2]
  \arrow[l]
  \arrow[d,hook]
  &
  \cC[1]\arrow[l,hookrightarrow]
  \arrow[dll, phantom, "\text{\reflectbox{$\pushout$}}",
  very near end, yshift=0.2cm, xshift=-0.1cm]
  \\
  &
  &
  \Cone\cC[1]
  &
  \cC[2]\arrow[l]
  &
\end{tikzcd}
\end{flushleft}
\end{minipage}
\begin{minipage}[t]{0.5\textwidth}
\begin{flushleft}
\begin{align*}
  \Cone\cC[0] = 
        \begin{boxedminipage}[t][2.0cm][t]{0.5\textwidth}
			\begin{tikzcd}[ampersand replacement=\&]
            \& \color{yellow!50!black}{\top} \&\\
            \phantom{\color{magenta!90!red}{\bullet}}
				\&\& \color{red!50!white}{\bullet}
                \arrow[lu, orange, very thick]
			\end{tikzcd}
            \end{boxedminipage}
            \\
	\cA[1] = 
           \begin{boxedminipage}[t][2.0cm][t]{0.5\textwidth}
            \begin{tikzcd}[ampersand replacement=\&]
            \& \color{yellow!50!black}{\top} \&\\
				\color{magenta!90!red}{\bullet}
				\&\& \color{red!50!white}{\bullet}
                \arrow[lu, orange, very thick]
			\end{tikzcd}
            \end{boxedminipage}
            \\
	\cA'[1] =
        \begin{boxedminipage}[t][2.0cm][t]{0.5\textwidth}
			\begin{tikzcd}[ampersand replacement=\&]
            \& \color{yellow!50!black}{\top} \&\\
				\color{magenta!90!red}{\bullet}
                \arrow[rr, blue, very thick, ""{name=D}]
				\&\& \color{red!50!white}{\bullet}
                \arrow[lu, orange, very thick]
			\end{tikzcd}
            \end{boxedminipage}
            \\
    \cA''[1] = 
    \begin{boxedminipage}[t][2.0cm][t]{0.5\textwidth}
			\begin{tikzcd}[ampersand replacement=\&]
            \& \color{yellow!50!black}{\top} \&\\
				\color{magenta!90!red}{\bullet}
                \arrow[ru, green, very thick] 
                \arrow[rr, blue, very thick, ""{name=D}]
				\&\& \color{red!50!white}{\bullet}
                \arrow[lu, orange, very thick]
			\end{tikzcd}
            \end{boxedminipage}
            \\
    \Cone\cC[1] = 
        \begin{boxedminipage}[t][2.0cm][t]{0.5\textwidth}
			\begin{tikzcd}[ampersand replacement=\&]
            \& \color{yellow!50!black}{\top} \&\\
				\color{magenta!90!red}{\bullet}
                \arrow[ru, green, very thick, ""{name=G, below}] 
                \arrow[rr, blue, very thick, ""{name=D}]
				\&\& \color{red!50!white}{\bullet}
                \arrow[lu, orange, very thick, ""{name=O}]
				\ar[
                    Rightarrow,
                    to=O,
                    from=G,
                    shorten >=5pt,
                    shorten <=5pt,
                    cyan,
                    ""
                ]
			\end{tikzcd}
            \end{boxedminipage}
\end{align*}
\end{flushleft}
\end{minipage}
\end{rmk}

We can now prove the lemma:

\begin{proof}[Proof of \cref{ConeCell1}]
By \cite[Th\'eor\`eme~6.29]{AraMaltsiniotisJoin}, we know that $\Cone\cC[\ell]$ is a strong Steiner $\omega$-category (in the sense of \cite[\textsection 2.15]{AraMaltsiniotisJoin}).
As such, it admits a basis in the sense of \cite[Definition~1.21]{AGOR}, which can be explicitly understood in terms of its algebraic model $\lambda\Cone\cC[\ell]$ by \cite[Th\'eor\`eme~2.12]{AraMaltsiniotisJoin}.
Finally, by \cite[\textsection9.1]{AraMaltsiniotisJoin}, we have an explicit description of one such
basis for the algebraic model $\lambda\Cone\cC[\ell]$ of $\Cone\cC[\ell]$.
This basis consists of:
\begin{itemize}
    \item a $0$-cell
    \item for $j=0,\dots,\ell-1$, two $j$-cells and two $(j+1)$-cells,
and
    \item an $\ell$-cell and an $(\ell+1)$-cell.
\end{itemize}
Hence, we see by induction on $\ell>0$ that the basis of $\Cone\cC[\ell]$ is obtained from that of $\Cone\cC[\ell-1]$ by adding exactly
\begin{itemize}
    \item an $(\ell-1)$-cell and an $\ell$-cell, and
    \item an $\ell$-cell and an $(\ell+1)$-cell.
\end{itemize}
The first claim then follows, and it is then straightforward to check that the total composite is the desired induced map.
\end{proof}

The formula above for cones simplifies as follows:

\begin{lem}
\label{ConeCell2}
For all $\ell>0$, the cone $\Cone\cC[\ell]$ can be expressed as an iterated pushout of the form
\[
\begin{tikzcd}
\cC[\ell-1]\arrow[dr, phantom, "\pushout", very near end, yshift=0.2cm, xshift=0.1cm]\arrow[r]\arrow[d,hook]\arrow[d]&\Cone\cC[\ell-1]\arrow[d]&\\
\cC[\ell]\arrow[r]&\cA'[\ell]\arrow[d]&\cC[\ell]\arrow[dl, phantom, "\text{\reflectbox{$\pushout$}}", very near end, yshift=0.2cm, xshift=0.1cm]\arrow[l]\arrow[d,hook]\\
&\Cone\cC[\ell]&\cC[\ell+1]\arrow[l]\\
\end{tikzcd}
\]
Moreover, the vertical composite
\[
\Cone\cC[\ell-1]\to\cA'[\ell]\to\Cone\cC[\ell]
\]
is induced by the cotarget map $\cC[\ell-1]\to\cC[\ell]$.
\end{lem}

To provide intuition, let's unpack the construction for the case $\ell=1$.

\begin{rmk}
The iterated pushout from \cref{ConeCell2} for $\ell=1$ can be described as
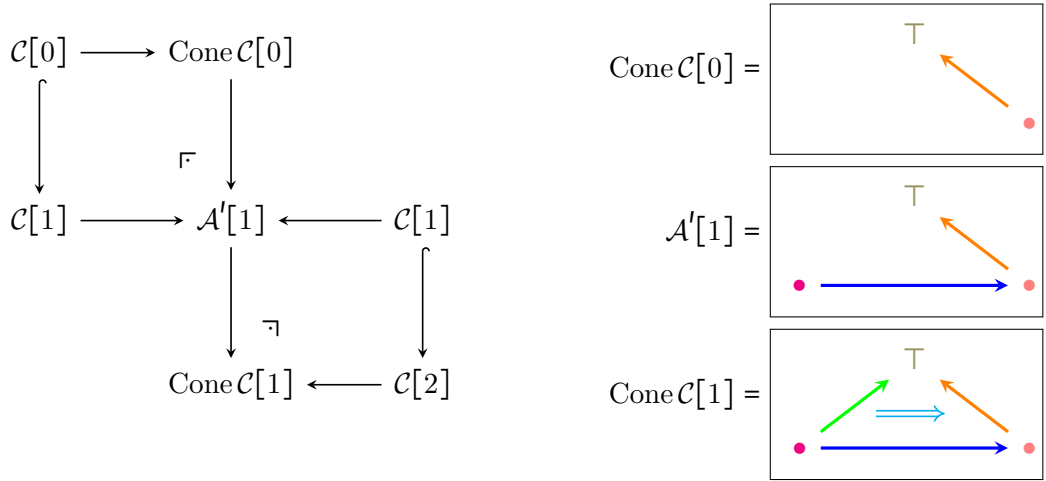
\begin{minipage}[t]{0.5\textwidth}
\begin{flushleft}
 \vspace{0.5cm}
\begin{tikzcd}[column sep=1cm, row sep=1.48cm]
  \cC[0]
  \arrow[r]
  \arrow[d,hook]
  \arrow[dr, phantom, "\pushout", very near end,
  yshift=0.2cm, xshift=0.1cm]
  &
  \Cone\cC[0]\arrow[d]
  &
  \\
  \cC[1]
  \arrow[r]
  &
  \cA'[1]\arrow[d]
  &
  \cC[1]
  \arrow[l]
  \arrow[d,hook]
  \arrow[dl, phantom, "\text{\reflectbox{$\pushout$}}",
  very near end, yshift=0.2cm, xshift=-0.1cm]
  \\
  &
  \Cone\cC[1]
  &
  \cC[2]\arrow[l]
\end{tikzcd}
\end{flushleft}
\end{minipage}
 \begin{minipage}[t]{0.5\textwidth}
 \begin{flushleft}
\begin{align*}
  \Cone\cC[0] = 
        \begin{boxedminipage}[t][2.0cm][t]{0.5\textwidth}
			\begin{tikzcd}[ampersand replacement=\&]
            \& \color{yellow!50!black}{\top} \&\\
				\phantom{\color{magenta!90!red}{\bullet}}
				\&\& \color{red!50!white}{\bullet}
                \arrow[lu, orange,  very thick]
			\end{tikzcd}
            \end{boxedminipage}
            \\
	\cA'[1] = 
           \begin{boxedminipage}[t][2.0cm][t]{0.5\textwidth}
            \begin{tikzcd}[ampersand replacement=\&]
            \& \color{yellow!50!black}{\top} \&\\
				\color{magenta!90!red}{\bullet}
                \arrow[rr, blue, very thick, ""{name=D}]
				\&\& \color{red!50!white}{\bullet} 
                \arrow[lu, orange,  very thick]
			\end{tikzcd}
            \end{boxedminipage}
            \\
    \Cone\cC[1] = 
        \begin{boxedminipage}[t][2.0cm][t]{0.5\textwidth}
			\begin{tikzcd}[ampersand replacement=\&]
            \& \color{yellow!50!black}{\top} \&\\
				\color{magenta!90!red}{\bullet}\arrow[ru, green, very thick, ""{name=G, below}] 
                \arrow[rr, blue, very thick, ""{name=D}] 
				\&\& \color{red!50!white}{\bullet} 
                \arrow[lu, orange,  very thick, ""{name=O}]
				\ar[Rightarrow, from=G, to=O, 
                shorten >= 5pt, shorten <= 5pt,  cyan,
                ""]
			\end{tikzcd}
            \end{boxedminipage}
\end{align*}
\end{flushleft}
\end{minipage}
\end{rmk}

We can now prove the lemma:

\begin{proof}[Proof of \cref{ConeCell2}]
Consider the following two commutative diagrams:
\[
\begin{tikzcd}[column sep=small]
  \partial\cC[\ell-1]
  \arrow[r]
  \arrow[d,hook]
  &
  \cC[\ell-1]
  \arrow[r]
  \arrow[d]
  &
  \Cone\cC[\ell-1]\arrow[d]
  \\
  \cC[\ell-1]
  \arrow[r,hookrightarrow]
  &
  \partial\cC[\ell]
  \arrow[d,hook]
  \arrow[r]
  &
  \cA[\ell]\arrow[d]
  \\
  &
  \cC[\ell]\arrow[r]
  &
  \cA'[\ell].
\end{tikzcd}
\qquad
\begin{tikzcd}[column sep=small]
  \partial\cC[\ell]
  \arrow[r]
  \arrow[d,hook]
  &
  \cC[\ell]
  \arrow[r]
  \arrow[d]
  &
  \cA'[\ell]\arrow[d]
  \\
  \cC[\ell]
  \arrow[r,hookrightarrow]
  &
  \partial\cC[\ell+1]
  \arrow[d,hook]
  \arrow[r]
  &
  \cA''[\ell]\arrow[d]
  \\
  &
  \cC[\ell+1]\arrow[r]
  &
  \Cone\cC[\ell]
\end{tikzcd}
\]
We know by \cref{ConeCell1} that the upper rectangles are pushouts, and the left top squares are pushouts by definition.
By pushout cancellation, so are the upper right squares.
By pushout pasting, so are the right rectangles. These are precisely the two pushout squares asserted in the statement. The first claim then follows, and it is then straightforward to check that the total composite is the desired induced map.
\end{proof}

\begin{thm}
\label{ConeCell3}
For all $k\geq0$, the cone $\Cone\cC[k]$ can be expressed as an iterated pushout of the form
\[
  \begin{tikzcd}
\cC[k+2]
\arrow[r]
\arrow[d]
\arrow[dr, phantom, "\pushout", very near end,
yshift=0.2cm, xshift=0.1cm]
&
\Cone\cC[k+1]\arrow[d]
&
\\
\cC[k+1]
\arrow[r]
&
\cA'[k+1]\arrow[d]
&
\cC[k+1]
\arrow[l]
\arrow[d]
\arrow[dl, phantom, "\text{\reflectbox{$\pushout$}}",
very near end, yshift=0.2cm, xshift=0.1cm]
\\
&
\Cone\cC[k]
&
\cC[k]\arrow[l]
\end{tikzcd}
\]
Moreover, the vertical composite
\[
\Cone\cC[k+1]\to\cA'[k+1]\to\Cone\cC[k]
\]
is induced by the canonical retraction $\cC[k+1]\to\cC[k]$.
\end{thm}

To provide intuition, let's unpack the construction for the case $k=0$.

\begin{rmk}
The iterated pushout from \cref{ConeCell1} for $k=0$
can be described as
\begin{minipage}[t]{0.5\textwidth}
\begin{flushleft}
 \vspace{0.5cm}
\begin{tikzcd}[column sep=1cm, row sep=1.48cm]
\cC[2]
\arrow[r]
\arrow[d]
\arrow[dr, phantom, "\pushout", very near end,
yshift=0.2cm, xshift=0.1cm]
&
\Cone\cC[1]\arrow[d]
&
\\
\cC[1]
\arrow[r]
&
\cA'[1]\arrow[d]
&
\cC[1]
\arrow[l]
\arrow[d]
\arrow[dl, phantom, "\text{\reflectbox{$\pushout$}}",
very near end, yshift=0.2cm, xshift=-0.1cm]
\\
&
\Cone\cC[0]
&
\cC[0]\arrow[l]
\\
\end{tikzcd}
\end{flushleft}
\end{minipage}
 \begin{minipage}[t]{0.5\textwidth}
 \begin{flushleft}
\begin{align*}
\Cone\cC[1] = \begin{boxedminipage}[t][2.0cm][t]{0.5\textwidth}
			\begin{tikzcd}[ampersand replacement=\&]
            \& \color{yellow!50!black}{\top} \&\\
				\color{magenta!90!red}{\bullet}\arrow[ru, green, very thick, ""{name=G, below}] 
                \arrow[rr, blue, very thick, ""{name=D}]
				\&\& \color{red!50!white}{\bullet} 
                \arrow[lu, orange,  very thick, ""{name=O}]
				\ar[Rightarrow, from=G, to=O, 
                shorten >= 5pt, shorten <= 5pt,  cyan,
                ""]
			\end{tikzcd}
            \end{boxedminipage}
            \\
	\cA'[1] = 
           \begin{boxedminipage}[t][2.0cm][t]{0.5\textwidth}
            \begin{tikzcd}[ampersand replacement=\&]
            \& \color{yellow!50!black}{\top} \&\\
				\color{magenta!90!red}{\bullet}
                \arrow[rr, blue, very thick, ""{name=D}]
				\&\& \color{red!50!white}{\bullet} 
                \arrow[lu, orange,  very thick]
			\end{tikzcd}
            \end{boxedminipage}
            \\
      \Cone\cC[0] = \begin{boxedminipage}[t][2.0cm][t]{0.5\textwidth}
			\begin{tikzcd}[ampersand replacement=\&]
            \& \color{yellow!50!black}{\top} \&\\
				\phantom{\color{magenta!90!red}{\bullet}}
				\&\& \color{red!50!white}{\bullet}
                \arrow[lu, orange,  very thick]
			\end{tikzcd}
            \end{boxedminipage}
\end{align*}
\end{flushleft}
\end{minipage}
\end{rmk}

\begin{lem}
\label{lem:retract-pushout}
Consider a pushout square
  \[
    \begin{tikzcd}
      X
      \ar[r]
      \ar[d,"s"swap]
      \arrow[dr, phantom, "\pushout", very near end, yshift=0.2cm, xshift=0.1cm]
      &
      X'
      \ar[d,"s'"]
      \\
      Y
      \ar[r]
      &Y'
    \end{tikzcd}
  \] 
  and let $r\colon Y\to X$ be a retraction of $s\colon X\to Y$.
  Then $s'$ has a unique retraction $r':Y'\to X$ that makes the square
  \[
    \begin{tikzcd}
      Y
      \ar[r]
      &Y'
      \\
      X
      \ar[r]
      \ar[from=u,"r"swap]
      &
      X'
      \ar[from=u,"r'",dashed]
    \end{tikzcd}
  \] 
commute and moreover this square is also a pushout square.
\end{lem}

\begin{proof}
  By the universal property of pushout there is a unique dashed map
  in the following commutative diagram
  \[
    \begin{tikzcd}
      X
      \ar[r]
      \ar[d,"s"]
      \ar[dd,"="', bend right]
      \arrow[dr, phantom, "\pushout", very near end, yshift=0.2cm, xshift=0.1cm]
      &
      X'
      \ar[d,"s'"]
      \ar[dd,"=", bend left=45]
      \\
      Y
      \ar[r]
      &Y'
      \\
      X
      \ar[r]
      \ar[from=u,"r"]
      &
      X'
      \ar[from=u,"r'",dashed]
    \end{tikzcd}
  \] 
  Moreover, the lower square is pullback by cancellation.
\end{proof}

We can now prove the main formula:

\begin{proof}[Proof of \cref{ConeCell3}]
By \cref{ConeCell3} we have pushout diagrams
\[
\begin{tikzcd}
\cC[k]\arrow[dr, phantom, "\pushout", very near end, yshift=0.2cm, xshift=0.1cm]\arrow[r]\arrow[d,hook]\arrow[d,hook]&\Cone\cC[k]\arrow[d,hook]\\
\cC[k+1]\arrow[r]&\cA'[k+1]
\end{tikzcd}
\quad\quad
\begin{tikzcd}
\cA'[k+1]\arrow[d,hook]&\cC[k+1]\arrow[d,hook]\arrow[l]\ar[ld,phantom,"\text{\reflectbox{$\pushout$}}",very near end,yshift=0.2cm, xshift=0.1cm]\\
\Cone\cC[k+1]&\cC[k+2]\arrow[l]
\end{tikzcd}
\]
Recall that the cotarget $\cC[k]\hookrightarrow\cC[k+1]$ and $\cC[k+1]\hookrightarrow\cC[k+2]$
admit retractions $r_k\colon\cC[k+1]\to\cC[k]$ and $r_{k+1}\colon\cC[k+2]\to\cC[k+1]$.
By \cref{lem:retract-pushout}, they induce retractions
\[
[r_k,\id]\colon\cA'[k+1]\to\Cone\cC[k]\quad\text{ and }\quad [\id,r_{k+1}]\colon\Cone\cC[k+1]\to\cA'[k+1]
\]
for the inclusions
\[
\Cone\cC[k]\hookrightarrow\cA'[k+1]\quad\text{ and }\quad \cA'[k+1]\hookrightarrow\Cone\cC[k+1],
\]
and there are pushout diagrams
\[
\begin{tikzcd}
\cC[k+1]
\arrow[dr, phantom, "\pushout", very near end, yshift=0.2cm, xshift=0.1cm]
\arrow[r]\arrow[d,"r_k"swap]\arrow[d,two heads]&
\cA'[k+1]
\arrow[d,two heads,"{[r_k,\id]}"]\\
\cC[k]\arrow[r]&\Cone\cC[k]
\end{tikzcd}
\quad\quad
\begin{tikzcd}
\Cone\cC[k+1]
\arrow[d,two heads,"{[\id,r_{k+1}]}"swap]&\cC[k+2]\arrow[d,two heads,"r_{k+1}"]\arrow[l]
\ar[ld,phantom,"\text{\reflectbox{$\pushout$}}",very near end,yshift=0.2cm, xshift=0.1cm]\arrow[l]
\\
\cA'[k+1]&\cC[k+1]\arrow[l]
\end{tikzcd}
\]
The first claim then follows, and it is then straightforward to check that the total composite is the desired induced map.
\end{proof}

\section{About $\omega$-categories - homotopy theory}

In this section we will assume the reader to be familiar with the basics of model category theory; see e.g.~\cite{hovey}.

\subsection{Background on the homotopy theory of $\omega$-categories}

\begin{rmk}
\label{ModelStructuresCat}
As proven in \cite[\textsection4,~\textsection6]{LMW},
\begin{itemize}[leftmargin=*]
\item there is a model structure $\omega\cat_{\mathrm{can}}$ on the category $\omega\cat$ of $\omega$-categories in which
\begin{itemize}
    \item every object is fibrant;
    \item the cofibrant objects are the polygraphs; these are $\omega$-categories obtained from the initial $\omega$-category by successively adjoining cells along the boundary inclusions $\partial\cC[n]\hookrightarrow\cC[n]$, in non-decreasing order of dimension.
    \item the set of boundary inclusions $\partial\cC[k]\hookrightarrow\cC[k]$ for $k\geq0$ generates the class of cofibrations
    \item given, for $k\geq0$, a factorization     $\partial\cC[k+1]\hookrightarrow\cP[k]\xrightarrow{\simeq}\cC[k]$ of the canonical map $\partial\cC[k+1]\to\cC[k]$ as a cofibration followed by a trivial fibration,
    the class of inclusions $\cC[k]\hookrightarrow\cP[k]$ induced by source inclusion $\cC[k]\hookrightarrow\partial\cC[k+1]$ for $k\geq0$ generates the class of acyclic cofibrations.
    \end{itemize}
    \item for $n\geq0$, there is a model structure $n\cat_{\mathrm{can}}$
    on the category $n\cat$ of $n$-categories in which every object is fibrant and the class of boundary inclusions $\partial\cC[k]\hookrightarrow\cC[k]$ for $0\leq k\leq n$ together with the folding map $\partial\cC[n+1]\to\cC[n]$ generates the class of cofibrations. 
\end{itemize}
\end{rmk}

\begin{prop}
\label{SigmaPreservesCofibrations}
For all $n\geq0$, the functor
$\Sigma^n\colon\omega\cat_{\mathrm{can}}\to\omega\cat_{\mathrm{can}}$ preserves cofibrations.
\end{prop}

\begin{proof}
First, we observe that the functor $\Sigma\colon\omega\cat_{\mathrm{can}}\to{}^{\partial\cC[1]/}\omega\cat_{\mathrm{can}}$ is a left adjoint functor which preserves generating cofibrations; hence, it preserves all cofibrations. Moreover, the forgetful functor $U\colon {}^{\partial\cC[1]/}\omega\cat_{\mathrm{can}}\to\omega\cat_{\mathrm{can}}$ preserves (in fact, creates) cofibrations by definition. It follows that the functor $\Sigma\colon\omega\cat_{\mathrm{can}}\to\omega\cat_{\mathrm{can}}$ preserves cofibrations, and that the iterated composite $\Sigma^n\colon\omega\cat_{\mathrm{can}}\to\omega\cat_{\mathrm{can}}$ preserves cofibrations, concluding the proof.
\end{proof}

We record a more explicit description for a set of generating acyclic cofibrations in $\omega\cat_{\mathrm{can}}$.

\begin{lem}
\label{AcyclicCofOmegaCat0}
Let $\cP$ be any $\omega$-category that is contractible and cofibrant in $\omega\cat_{\mathrm{can}}$
and has at least two objects. Given any injective-on-objects
map $\partial\cC[1]\to\cP$, the following hold.
\begin{enumerate}[leftmargin=*, ref=(\arabic*)]
\item\label{SuspPoly1} For all $n\geq0$ the induced map $\Sigma^n\partial\cC[1]\to\Sigma^n\cP$ is a cofibration in $\omega\cat_{\mathrm{can}}$.
\item\label{SuspPoly2} For all $n\geq0$ the induced map $\Sigma^n\cP\to\Sigma^n\cC[0]$ is an acyclic fibration in $\omega\cat_{\mathrm{can}}$.
\end{enumerate}
\end{lem}

\begin{proof}
We first show \ref{SuspPoly1} by induction on $n\geq0$; that is, the map $\Sigma^n\partial\cC[1]\to\Sigma^n\cP$ is a cofibration in $\omega\cat_{\mathrm{can}}$.
For $n=0$, we observe that the map $\partial\cC[1]\to\cP$ factors as
\[
\partial\cC[1]\hookrightarrow \mathrm ob\cP\hookrightarrow \cP.
\]
The first map $\partial\cC[1]\hookrightarrow \mathrm ob\cP$ is a cofibration, since it is obtained by adjoining the remaining objects of $\cP$, while the second map $\mathrm ob\cP\hookrightarrow \cP$ is a cofibration, since $\cP$ is a polygraph and is obtained from its set of objects by successively adjoining its positive-dimensional generators. Hence, the map $\partial\cC[1]\hookrightarrow\cP$ is a cofibration.
For $n>0$, the fact that $\Sigma^n\partial\cC[1]\to\Sigma^n\cP$ is a cofibration in $\omega\cat_{\mathrm{can}}$ follows from \cref{SigmaPreservesCofibrations}, proving \ref{SuspPoly1}.

We now show \ref{SuspPoly2} by induction on $n\geq0$; that is, the map $\Sigma^n\cP\to\Sigma^n\cC[0]$ is an acyclic fibration in $\omega\cat_{\mathrm{can}}$.
For $n=0$, we know that $\cP\to\cC[0]$ is a weak equivalence because $\cP$ is contractible and that it is a fibration because $\cP$ (like any other object) is fibrant on $\omega\cat_{\mathrm{can}}$. For $n>0$, assume we are given the lifting problem in $\omega\cat$:
\[
\begin{tikzcd}
\partial \cC[k] \arrow[r]\arrow[d]& \Sigma^n \cP\arrow[d]\\
\cC[k] \arrow[r]\arrow[ru,dashed,"?"]& \Sigma^n \cC[0].
\end{tikzcd}
\]
We note that the right vertical map is a bijection on objects. Then, if $k=0$, a lift exists. If $k>0$, thanks to the identifications
\[
\partial\cC[k]\cong \Sigma\partial\cC[k-1]
\qquad\text{and}\qquad
\cC[k]\cong \Sigma\cC[k-1],
\]
the left vertical map is in the image of $\Sigma$,
we can upgrade it to a lifting problem in ${}^{\partial\cC[1]/}\omega\cat$:
\[
\begin{tikzcd}
(\Sigma\partial \cC[k-1],\bot,\top) \arrow[r]\arrow[d]& (\Sigma^n \cP,a,b)\arrow[d]\\
(\Sigma\cC[k-1],\bot,\top) \arrow[r]\arrow[ru,dashed,"?"]& (\Sigma^n \cC[0],a,b).
\end{tikzcd}
\]
Transposing along the adjunction from \cref{SuspensionAdjunctionCat} we can rewrite it as a lifting problem in $\omega\cat$ for some $a$ and $b$ objects in $\Sigma^n\cP$:
\[
\begin{tikzcd}
    \partial \cC[k-1] \arrow[r]\arrow[d]& (\Sigma^n \cP)(a,b)\arrow[d]\\
    \cC[k-1] \arrow[r]\arrow[ru,dashed,"?"]& (\Sigma^n \cC[0])(a,b).
\end{tikzcd}
\]
If $a\neq\bot$ or $b\neq\top$, the right vertical map is an isomorphism, so a lift exists. Otherwise, the lifting problem becomes
\[
\begin{tikzcd}
    \partial \cC[k-1] \arrow[r]\arrow[d]& \Sigma^{n-1} \cP\arrow[d]\\
    \cC[k-1] \arrow[r]\arrow[ru,dashed,"?"]& \Sigma^{n-1}\cC[0].
\end{tikzcd}
\]
and a lift exists by induction hypothesis, proving the claim.
\end{proof}

\begin{prop}
\label{AcyclicCofOmegaCat}
Given a factorization
$\partial \cC[1] \hookrightarrow \cP \xrightarrow{\simeq} \cC[0]$ 
of the canonical map $\partial\cC[1]\to\cC[0]$ as a cofibration followed by a weak equivalence, 
the set of inclusions $\Sigma^n\cC[0]\hookrightarrow\Sigma^n\cP$
induced by the source inclusion $\cC[0]\hookrightarrow\partial\cC[1]$
for $n\geq0$ generates the class of acyclic cofibrations of the model category $\omega\cat_{\mathrm{can}}$.
\end{prop}

\begin{proof}
Applying $\Sigma^n$ for $n\geq0$ to the given factorization we obtain
\[
\partial \cC[n+1]\cong\Sigma^n\partial \cC[1] \hookrightarrow \Sigma^n\cP \xrightarrow{\simeq} \Sigma^n\cC[0]\cong\cC[n].
\]
By \cref{AcyclicCofOmegaCat0}, this is a factorization
of the map $\partial\cC[n+1]\to\cC[n]$ as a cofibration followed by a weak equivalence.
Hence, the set of maps $\Sigma^n\cC[0]\hookrightarrow\Sigma^n\cP$ is by \cref{ModelStructuresCat} a set of generating acyclic cofibrations for $\omega\cat_{\mathrm{can}}$, as desired.
\end{proof}

\begin{prop}
\label{SuspensionQuillenCat}
The suspension-hom pair from \cref{SuspensionAdjunctionCat} forms a Quillen adjunction
\[\Sigma\colon\omega\cat_{\mathrm{can}}\rightleftarrows{}^{\partial\cC[1]/}\omega\cat_{\mathrm{can}}\noloc(-)(-,-),\]
where the undercategory ${}^{\partial\cC[1]/}\omega\cat_{\mathrm{can}}$ is equipped with the induced model structure of \cite{HirschhornOvercategories}.
\end{prop}

\begin{proof}
By definition, the suspension $\Sigma$ preserves generating cofibrations and, since it is a left adjoint by \cref{SuspensionAdjunctionCat}, all cofibrations. By \cref{AcyclicCofOmegaCat}, the suspension $\Sigma$ preserves a set of generating acyclic cofibrations and since it is a left adjoint, all acyclic cofibrations.
\end{proof}

\begin{prop}
\label{IteratedSuspensionQuillenCat}
The Quillen adjunction from \cref{SuspensionQuillenCat} induces a Quillen adjunction
\[\Sigma^{\ell}\colon\omega\cat_{\mathrm{can}}\rightleftarrows{}^{\partial\cC[\ell]/}\omega\cat_{\mathrm{can}}\noloc(-)^{(\ell)}(-),\]
where the undercategory ${}^{\partial\cC[\ell]/}\omega\cat_{\mathrm{can}}$ is equipped with the induced model structure of \cite{HirschhornOvercategories}.
\end{prop}

\begin{proof}
This is an instance of \cref{IteratedQuillenAdjunction}, applied to
$\cC=\omega\cat_{\mathrm{can}}$, $c_0=\varnothing$, $F=\Sigma$, and $G=(-)^{(-)}(-)$,
using the suspension-hom Quillen adjunction of \cref{SuspensionQuillenCat}.
\end{proof}

\begin{lem}
\label{IteratedQuillenAdjunction}
Let $\cM$ be a model category, let $F\colon\cM\to\cM$ be a functor, and let $c_0$ be an initial object of $\cM$. If $F$ induces a Quillen adjunction
\begin{equation}
\label{eq:original-quillen-adjunction}
F\colon
\cM
\simeq
{}^{c_0/}\cM
\rightleftarrows
{}^{F(c_0)/}\cM
\noloc G,
\end{equation}
then for all $\ell\geq0$ there is a Quillen adjunction
\[
F^\ell\colon\cM\simeq {}^{c_0/}\cM
\rightleftarrows
{}^{F(c_0)/}\cM
\rightleftarrows
{}^{F^2(c_0)/}\cM
\rightleftarrows
\cdots
\rightleftarrows
{}^{F^\ell(c_0)/}\cM\noloc G^\ell.
\]
where, for $\ell\geq0$, the undercategory ${}^{F^\ell(c_0)/}\cM$ is equipped with the induced model structure of \cite{HirschhornOvercategories}.
\end{lem}

\begin{proof}
We show this by induction on $\ell\geq0$. For $\ell=0$, using that $c_0$ is initial, this is just the canonical Quillen equivalence
\[
\cM\simeq{}^{c_0/}\cM,
\]
and for $\ell=1$ this is the given Quillen adjunction \eqref{eq:original-quillen-adjunction}. Now, by induction, assume that we have constructed a Quillen adjunction
\begin{equation}
\label{eq:quilleninduction}
F^\ell\colon\cM\rightleftarrows{}^{F^\ell(c_0)/}\cM\noloc G^\ell.
\end{equation}
Furthermore, by taking the slice of the original Quillen adjunction
\eqref{eq:original-quillen-adjunction}
under the object $F^\ell(c_0)$ of $\cM$, we obtain a Quillen adjunction
\begin{equation}
\label{eq:sliced-quillen-adjunction}
{}^{F^\ell(c_0)/}\cM\rightleftarrows{}^{F(F^\ell(c_0))/}\left({}^{F(c_0)/}\cM\right).
\end{equation}
Using the identification $F(F^\ell(c_0))=F^{\ell+1}(c_0)$,
we also obtain a canonical Quillen equivalence
\begin{equation}
\label{eq:slice-equivalence-quillen}
{}^{F(F^\ell(c_0))/}\left({}^{F(c_0)/}\cM\right)\simeq{}^{F^{\ell+1}(c_0)/}\cM.
\end{equation}
Composing the Quillen adjunctions
\eqref{eq:quilleninduction}, \eqref{eq:sliced-quillen-adjunction}, and \eqref{eq:slice-equivalence-quillen}, we obtain the Quillen adjunction
\[
F^{\ell+1}\colon \cM \rightleftarrows {}^{F^{\ell+1}(c_0)/}\cM \noloc G^{\ell+1},
\]
as desired.
\end{proof}

\begin{rmk}
\label{TruncationQuillenCat}
It follows from \cite[Theorem 6.1]{LMW} that the adjunction
\[
\mathrm{tr}_n\colon \omega\cat\rightleftarrows n\cat\noloc I_n
\]
is a Quillen adjunction. Instead, the adjunction
\[
I_n\colon n\cat\rightleftarrows\omega\cat\noloc\mathrm{core}_n
\]
is not. In fact, the core functor $\mathrm{core}_n$ does not preserve weak equivalences.
\end{rmk}

\subsection{Contractible gaunt $\omega$-categories}
\label{ContractibleGaunt}

In this subsection we collect several properties of $\omega$-categories which are both contractible and gaunt, whose definitions we now recall.

\begin{defn}
  \label{defn:contr-gaunt}
Given an $\omega$-category $\cD$, we say that it is:
    \begin{itemize}[leftmargin=*]
        \item \emph{contractible} if $\cD\to[0]$ is a weak equivalence in the model structure $\omega\cat_{\mathrm{can}}$.
        Equivalently, $\cD$ is contractible if it has the right lifting property with respect to the cell boundary inclusions $\partial\cC[k]\hookrightarrow\cC[k]$ in $\omega\cat$ for all $k\geq0$.
      \item \emph{gaunt} if for all $k\geq0$ the $k$-core $\mathrm{core}_k\cD$ is a gaunt $k$-category; that is, $\mathrm{core}_k\cD$ has no non-identity isomorphisms.
Equivalently, $\cD$ is gaunt if it has the right lifting property with respect to the canonical map $\Sigma^k\cI\to\Sigma^k[0]\cong\cC[k]$ in $\omega\cat$ for all $k\geq0$.
    \end{itemize}
\end{defn}

\begin{prop}
\label{PolygraphIsGaunt}
If an $\omega$-category $\cD$ is a polygraph, then $\cD$ is gaunt.
\end{prop}

\begin{proof}
To show that $\cD$ is gaunt, it suffices to solve the following lifting problems in $\omega\cat$ for all $k\geq0$:
\[\begin{tikzcd}
\Sigma^k\cI\arrow[r,"{}"]\arrow[d]&\cD\arrow[d]\\
\Sigma^k[0]\arrow[ru,dashed,"?"]\arrow[r]&{[0]}
\end{tikzcd}
\]
Using the adjunction
from \cref{TruncationsAdjunctionCat},
we see that the lifting problem is equivalent to the lifting problem in $(k+1)\cat$
\[\begin{tikzcd}
\Sigma^k\cI\arrow[r,"{}"]\arrow[d]&\mathrm{core}_{k+1}\cD\arrow[d]\\
\Sigma^k[0]\arrow[ru,dashed,"?"]\arrow[r]&{[0]}
\end{tikzcd}
\]
Since $\cD$ is an $\omega$-category generated by a polygraph, and using the fact that the core functor preserves pushouts from \cref{ChaosAdjunctionCat}, we see that $\mathrm{core}_{k+1}\cD$ is a polygraph.
Then $\mathrm{core}_{k+1}\cD$ admits a $(k+1)$-basis $E_{k+1}$ in the sense of  \cite[Definition~1.2.1]{GuettaThesis}.
By \cite[Proposition~1.3.5]{GuettaThesis}, for all $\alpha$ in $E_{k+1}$ there exists a function
$w_\alpha\colon\cD_{k+1}\to\mathbb N$
which essentially counts the occurrences of the basis element $\alpha$ in the expression of a given $(k+1)$-cell of $\cD$, and has in particular the following properties:
\begin{enumerate}
    \item Given a $(k+1)$-cell $f$ of $\cD$, we have that $w_\alpha(f)=0$ for all $\alpha$ in $E_{k+1}$ if and only if $f$ is an identity $(k+1)$-cell (cf.~\cite[Remark 1.8.13]{GuettaThesis}).
    \item Given composable $(k+1)$-cells $g$ and $h$ in $\cD$, we have that $w_{\alpha}(h\circ g)=w_{\alpha}(g)+w_{\alpha}(h)$.
\end{enumerate}
If $x$ denotes the cell determined by the canonical map $\cC[k+1]\hookrightarrow\Sigma^k\cI\to\cD$, we have that $x$ is invertible and
\[
x\circ x^{-1}=\mathrm{id}.
\]
Then, for all $\alpha$ in $E_{k+1}$, we have $w_\alpha(x)+ w_{\alpha}(x^{-1})=w_\alpha(\mathrm{id})$, and further
\[0\leq w_\alpha(x)=w_{\alpha}(\mathrm{id})-w_{\alpha}(x^{-1})
\leq w_{\alpha}(\mathrm{id})=0.\]
Then, $w_\alpha(x)=0$ for all $\alpha$ in $E_{k+1}$. Then $x$ must be an identity $(k+1)$-cell. This means that the desired lifting problem admits a lift, concluding the proof.
\end{proof}

\begin{prop}
\label{E:exists}
The $\omega$-category $\cE$ constructed in \cite[Construction~1.5.13]{ORsurvey}---further studied in \cite[\textsection1.4]{HLOR}---is an $\omega$-category which has two objects, is gaunt and is contractible.
\end{prop}

\begin{proof}
The fact that this $\omega$-category has two objects is by construction. It was shown in \cite[Theorem~1.33]{HLOR} that the $\omega$-category $\cE$ is contractible. By construction, the $\omega$-category $\cE$ is a polygraph, hence gaunt by \cref{PolygraphIsGaunt}.
\end{proof}

We will need the following properties of contractible $\omega$-categories.

\begin{lem}
\label{E:hom}
If $\cD$ is a gaunt (resp.~contractible) $\omega$-category,
then for all objects $x,y\in\cD$,
the hom-$\omega$-category $\cD(x,y)$ is gaunt
(resp.~contractible).
\end{lem}

\begin{proof}
Assume that $\cD$ is gaunt (resp.~contractible).
To show that $\cD(x,y)$ is gaunt (resp.~contractible) we need to show that we can solve the following lifting problem in $\omega\cat$ for all $k\geq0$:
\[
\begin{tikzcd}
\Sigma^k\cI\arrow[r,"{}"]\arrow[d]&\cD(x,y)\arrow[d]\\
\Sigma^k[0]\arrow[ru,dashed,"?"]\arrow[r]&{[0]}
\end{tikzcd}
\quad\quad
\left(\text{resp.~}
\begin{tikzcd}
\partial\cC[k]\arrow[r,"{}"]\arrow[d,hook]&\cD(x,y)\arrow[d]\\
\cC[k]\arrow[ru,dashed,"?"]\arrow[r]&{[0]}
\end{tikzcd}
\right)
\]
Using the adjunction from \cref{SuspensionAdjunctionCat},
we see that the two lifting problems are equivalent to the following lifting problem in ${}^{\partial\cC[1]/}\omega\cat$
\[
\begin{tikzcd}
(\Sigma^{k+1}\cI,\bot,\top)\arrow[r,"{}"]\arrow[d]&(\cD,x,y)\arrow[d]\\
(\Sigma^{k+1}[0],\bot,\top)\arrow[ru,dashed,"?"]\arrow[r]&{([0],0,0)}
\end{tikzcd}
\quad\quad
\left(\text{resp.~}
\begin{tikzcd}
(\partial\cC[k+1],\bot,\top)\arrow[r,"{}"]\arrow[d,hook]&(\cD,x,y)\arrow[d]\\
(\cC[k+1],\bot,\top)\arrow[ru,dashed,"?"]\arrow[r]&{([0],0,0)}
\end{tikzcd}
\right)
\]
Since the left vertical maps in both diagrams are bijective on objects,
we see that it
suffices to solve it in $\omega\cat$
\[
\begin{tikzcd}
\Sigma^{k+1}\cI\arrow[r,"{}"]\arrow[d]&\cD\arrow[d]\\
\Sigma^{k+1}[0]\arrow[ru,dashed,"?"]\arrow[r]&{[0]}
\end{tikzcd}
\quad\quad
\left(\text{resp.~}
\begin{tikzcd}
\partial\cC[k+1]\arrow[r,"{}"]\arrow[d,hook]&\cD\arrow[d]\\
\cC[k+1]\arrow[ru,dashed,"?"]\arrow[r]&{[0]}
\end{tikzcd}
\right)
\]
This lifting problem can be solved because $\cD$ is by assumption gaunt (resp.~contractible), concluding the proof.
\end{proof}

\begin{lem}
\label{E:Parallel}
An $\omega$-category $\cD$ is contractible if and only if
for all $f$ and $g$ parallel $(k-1)$-morphisms\footnote{with the convention that ``a pair of parallel ($-1$)-arrows'' amounts to no data} in $\cD$ for $k\geq0$ there exists a $k$-morphism $f\to g$ in $\cD$.
\end{lem}

\begin{proof}
The $\omega$-category $\cD$ is contractible if and only if the map $\cD\to\cC[0]$ is an acyclic fibration; that is, for all $k\geq0$
there is a solution to the lifting problem in $\omega\cat$
\[
\begin{tikzcd}
\partial\cC[k]\arrow[r,"{[f,g]}"]\arrow[d,hook]&\cD\arrow[d,"\simeq"]\\
\cC[k]\arrow[ru,dashed,"?"]\arrow[r]&{[0]}
\end{tikzcd}
\]
This is equivalent to saying that, for all $f$ and $g$ parallel $(k-1)$-morphisms in $\cD$ there exists a $k$-morphism $f\to g$ in $\cD$, concluding the proof.
\end{proof}

\section{About marked simplicial sets}

\label{MarkedSimplicialSets}

\subsection{Background on marked simplicial sets}

A \emph{marked simplicial set} $X$ consists of a simplicial set endowed with a collection of subsets $tX_k\subseteq X_k$ (sometimes also denoted $X_{[k]_t}\subseteq X_k$) of \emph{marked} $k$-simplices
for $k>0$ containing all degenerate $k$-simplices. Let $m\sset$ denote the category of marked simplicial sets (originally considered in \cite{VerityComplicialI,VerityComplicialAMS} under the name of \emph{stratified simplicial sets}) and simplicial maps that preserve the marking.

Relevant constructions on marked simplicial sets (about varying $n$):

\begin{notn}
\label{TruncationsMarked}
Let $X$ be a marked simplicial set.
\begin{itemize}[leftmargin=*]
    \item For $n\geq0$, the marked simplicial set $\mathrm{th}_nX$ is the $n$-th \emph{thinification} of $X$ from \cite[\textsection3]{VerityComplicialAMS} or \cite[Notation 13]{VerityComplicialI}. It is obtained from $X$ by further marking all simplices in dimensions higher than $n$.
    This construction defines a functor $\mathrm{th}_n\colon m\sset\to m\sset$.
    \item For $n\geq0$, the marked simplicial set $\mathrm{sp}_nX$ is the $n$-th \emph{superstructure} of $X$ from \cite[\textsection 5.2]{VerityComplicialAMS} or \cite[Notation 13]{VerityComplicialI}.
    It is obtained from $X$ by forgetting all $d$-simplices that have a non-marked $k$-dimensional face with $n<k\leq d$.
    This construction defines a functor $\mathrm{sp}_n\colon m\sset\to m\sset$.
\end{itemize}
\end{notn}

\begin{rmk}
\label{TruncationsAdjunctionMarked}
We know from \cite[Definition 112]{VerityComplicialAMS} or \cite[Notation 13]{VerityComplicialI} that, for all $n\geq0$, there is an adjunction of categories
\[
\mathrm{th}_n\colon m\sset\rightleftarrows m\sset\noloc\mathrm{sp}_n.
\]
\end{rmk}

\begin{rmk}
\label{SpPreservesFilteredColimits}
For $n\geq0$, we see by direct inspection that the functor $\mathrm{sp}_n\colon m\sset\to m\sset$ preserves filtered colimits. Indeed, a simplex in a filtered colimit is marked precisely when it is represented by a simplex that becomes marked at some stage of the diagram.
\end{rmk}

\begin{rmk}
There are adjunctions of categories
\[
(-)^{\flat}\colon \sset\rightleftarrows m\sset\noloc U\quad\text{ and }\quad U\colon m\sset\rightleftarrows \sset\noloc (-)^{\sharp}.
\]
Here, the functor $U\colon m\sset\to\sset$ takes the underlying simplicial set of a marked simplicial set, while the functor $(-)^{\flat},(-)^{\sharp}
\colon \sset\to m\sset$ endow a simplicial set with the minimal marking (just degenerate simplices) and the maximal marking (all simplices in positive dimension), respectively. When clear from the context, for a simplicial set $X$ we denote $X^{\flat}$ just by $X$.
\end{rmk}

\begin{lem}
\label{ThOfEntire}
Given an \emph{entire} map $A \to A'$ in $m\sset$
(i.e., a map of marked simplicial sets whose underlying map of simplicial sets is an isomorphism),
there is a pushout in $m\sset$:
\[
\begin{tikzcd}
A \arrow[r]\arrow[d]& A'\arrow[d]\\
\mathrm{th}_nA\arrow[r] & \mathrm{th}_nA'
\end{tikzcd}
\]
\end{lem}

\begin{proof}
Given any $X$ marked simplicial set, we have that
\[
m\sset(\mathrm{th}_n A',X)=m\sset(A',X)\cap m\sset(\mathrm{th}_n A,X)
\]
It follows that the induced square of sets
\[
\begin{tikzcd}
m\sset(A,X)&m\sset(A',X)\ar[l]\\
m\sset(\mathrm{th}_nA,X)\ar[u]&m\sset(\mathrm{th}_nA',X)\ar[u]
\ar[l]
\end{tikzcd}
\]
is a pullback, and the claim follows.
\end{proof}

We discuss the two-point suspension construction in the context of marked simplicial sets.

\begin{notn}
\label{SuspensionMarked}
Let $X$ be a marked simplicial set.
\begin{itemize}[leftmargin=*]
    \item Given any marked simplicial set $Y$, the marked simplicial set $X\star Y$ is the \emph{join} of $X$ and $Y$ from \cite[\textsection3]{VerityComplicialI} and \cite[\textsection2]{ORfp}.
    This is given by endowing the join of the underlying simplicial sets
    of $X$ and $Y$ with the marking of all simplices $\sigma\star\tau$ for which $\sigma$ is marked in $X$ or $\tau$ is marked in $Y$.
    The join construction defines a functor $\star\colon m\sset\times m\sset\to m\sset$.
    \item The marked simplicial set $\Sigma X$ is the \emph{two-point suspension} of $X$ from
    \cite[\textsection2]{ORfp}, which is obtained as $\Sigma X\coloneqq (X\star\Delta[0])/X$.
    The suspension construction defines a cocontinuous functor $\Sigma\colon m\sset\to^{\partial\Delta[1]/}m\sset$.
    \item Given two vertices $a$ and $b$ of $X$, encoded as a map $(a,b)\colon\partial\Delta[1]\to X$, the marked simplicial set $X(a,b)$ is the \emph{hom} of $X$ at $(a,b)$ from  \cite[\textsection2]{ORfp}. The set of (marked) $k$-simplices of $X(a,b)$ is
given by
\[
(X(a,b))_{[k]_{(t)}}=\left({^{\partial\Delta[1]/}m\sset}\right)(\Sigma\Delta[k]_{(t)},X),
\]
The hom construction defines a functor $(-)(-,-)\colon{}^{\partial\Delta[1]/}m\sset\to m\sset$.
    \end{itemize}
\end{notn}

\begin{rmk}
\label{SuspensionAdjunctionMarked}
We know from \cite[\textsection2]{ORfp} that there is an adjunction
\[\Sigma\colon m\sset\rightleftarrows {}^{\partial\Delta[1]}/m\sset\noloc(-)(-,-),\]
where ${}^{\partial\Delta[1]}/m\sset$ 
denotes the category of bipointed marked simplicial sets.
\end{rmk}

We now introduce iterated hom-suspension:

\begin{prop}
\label{IteratedSuspensionAdjunctionMarked}
Let $\ell\geq0$. The adjunction from \cref{IteratedSuspensionAdjunctionMarked} induces an adjunction
\[\Sigma^{\ell}\colon m\sset\rightleftarrows{}^{\Sigma^{\ell}\varnothing/}m\sset\noloc(-)^{(\ell)}(-).\]
\end{prop}

\begin{proof}
This is an instance of \cref{IteratedAdjunction}, applied to
$\cC=m\sset$, $c_0=\varnothing$, $F=\Sigma$, and $G=(-)^{(-)}(-)$,
using the suspension-hom adjunction of \cref{SuspensionAdjunctionMarked}.
\end{proof}

\begin{rmk}
\label{GlobularBoundary}
Let $\ell\geq0$. Given marked simplicial sets $X,P$ and $\gamma\colon\Sigma^\ell X \to P$ a map of marked simplicial sets.
There is a canonical map of marked simplicial sets
\[\Sigma^\ell !\colon \Sigma^\ell\varnothing\to \Sigma^\ell X\]
and hence also a canonical map of marked simplicial sets
\[\glboundary\gamma\colon
\Sigma^\ell \varnothing \xrightarrow{\Sigma^\ell !} \Sigma^\ell X \xrightarrow{\gamma} P.\]
This makes $\gamma$ into a morphism $\Sigma^\ell X \to P$ in ${}^{\Sigma^{\ell}\varnothing/}m\sset$.
The transpose of this morphism with respect to the adjunction from \cref{IteratedSuspensionAdjunctionMarked} is then a morphism of marked simplicial sets
\[X \to P^{(\ell)}(\glboundary\gamma).\]
\end{rmk}

We conclude with a brief analysis of the simplices of iterated suspensions, which will be used later in the paper.

First, there is a comparison map between the iterated suspension and the join with a standard simplex:

\begin{rmk}
\label{ProgressiveSuspensionsGeneral}
For any marked simplicial set $X$ and $0\leq p\leq\ell$, there is a canonical and natural quotient map
\[
\Sigma^pX\star\Delta[\ell-p]
=
\Sigma^pX\star\Delta[0]\star\Delta[\ell-p-1]
\to
\Sigma(\Sigma^pX)\star\Delta[\ell-p-1]
=
\Sigma^{p+1}X\star\Delta[\ell-p-1].
\]
In total, there are canonical and natural quotient maps
\[
X\star\Delta[\ell]
\to
\Sigma X\star\Delta[\ell-1]
\to\dots\quad\quad\quad\quad\quad\quad\quad\quad\quad\quad\quad\quad\quad\quad\quad\quad\quad\quad\quad
\]
\[
\dots\to
\Sigma^pX\star\Delta[\ell-p]
\to
\Sigma^{p+1}X\star\Delta[\ell-p-1]
\to\dots
\]
\[
\quad\quad\quad\quad\quad\quad\quad\quad\quad\quad\quad\quad\quad\quad\quad\quad\quad
\dots\to
\Sigma^\ell X\star\Delta[0]
\to
\Sigma^{\ell+1}X.
\]
\end{rmk}

This comparison map can be used to access a description of the simplices of iterated suspensions:

\begin{lem}
\label{IteratedSuspensionMarked}
Given a marked simplicial set $X$, referring to \cref{ProgressiveSuspensionsGeneral}, every non-degenerate simplex of $\Sigma^\ell X$ is one of the following:
\begin{itemize}[leftmargin=1cm]
    \item[(A)] For $0\leq k\leq\ell-1$, the $k$-simplices
    \[\Sigma^{k}\bot\colon\Delta[k]\to\Sigma^k\Delta[0]\xrightarrow{\Sigma^k\bot}\Sigma^k\Sigma^{\ell-k}X
    \ \text{ and }\ 
    \Sigma^{k}\top\colon\Delta[k]\to\Sigma^k\Delta[0]\xrightarrow{\Sigma^k\top}\Sigma^k\Sigma^{\ell-k}X,\]
    which are not marked in $\Sigma^\ell X$.
    \item[(B)] For $\tau$ a non-degenerate $d$-simplex of $X$, the $(d+\ell)$-simplex
    \[\Sigma^\ell\tau\colon\Delta[d+\ell]\to\Sigma^\ell\Delta[d]
    \xrightarrow{\Sigma^\ell\tau}\Sigma^\ell X,\]
    which is marked $\Sigma^\ell X$ precisely if $\tau$ is marked in $X$.
\end{itemize}
\end{lem}

\begin{proof}
We see this by induction on $\ell\geq0$. When $\ell=0$, there is nothing to show. Now assume the statement true for $\ell$, and we show it for $\ell+1$. Using the induction hypothesis, we can first describe the non-degenerate simplices of $\Sigma^\ell(\Sigma X)$ in terms of those of $\Sigma X$, and then describe the non-degenerate simplices of $\Sigma X$ in terms of those of $X$.

Apply the induction hypothesis to $Y\coloneq \Sigma X$ to see that the non-degenerate simplices of $\Sigma^\ell Y$ are the following:
\begin{itemize}[leftmargin=1.5cm]
\item[($A_\ell(Y)$)] For each $0\leq k\leq \ell-1$ the $k$-simplices
\[\Sigma^{k}\bot\colon\Delta[k]\to\Sigma^k\Delta[0]
\xrightarrow{\Sigma^k\bot}\Sigma^k\Sigma^{\ell-k}Y
\ \text{ and }\ 
\Sigma^{k}\top\colon\Delta[k]\to\Sigma^k\Delta[0]
\xrightarrow{\Sigma^k\top}\Sigma^k\Sigma^{\ell-k}Y,\]
\item[($B_\ell(Y)$)]
\label{it:Y-tau-simplex}
For each non-degenerate $d$-simplex $\tau$ of $Y$,
the $(d+\ell)$-simplex
\[\Sigma^\ell\tau\colon\Delta[d+\ell]\to\Sigma^\ell\Delta[d]
\xrightarrow{\Sigma^\ell\tau}\Sigma^\ell Y,\]
which is marked $\Sigma^\ell Y$ precisely if $\tau$ is marked in $Y$.
\end{itemize}
So it remains to understand the non-degenerate $d$-simplices $\tau$ of $Y=\Sigma X$.
By the construction $Y = (X\star \Delta[0]) / X$
each such simplex $\tau$ is represented by one of the following:
\begin{itemize}
    \item $\varnothing\star \tau''\mapsto \tau$, which is non-degenerate precisely if $\dim\tau''=0$; in this case $\tau=\top$.
    \item $\tau'\star\varnothing \mapsto \tau$, which is non-degenerate precisely if $\dim\tau'=0$; in this case $\tau=\bot$.
    \item $\tau'\star\tau''\mapsto \tau$ which is non-degenerate precisely if $\dim\tau''=0$ and $\tau'$ is a non-degenerate $(d-1)$-simplex of $X$; in this case $\tau=\Sigma \tau'$.
      Moreover, $\tau$ is marked if and only if $\tau'$ is.
\end{itemize}

Thus the simplices of type ($B_\ell(Y)$)
can be subdivided into:
\begin{itemize}[leftmargin=1.5cm]
\item[($B_\ell(Y)'$)]
The two $(0+\ell)$-simplices
\[\Sigma^{\ell}\bot\colon\Delta[\ell]\to\Sigma^{\ell}\Delta[0]
\xrightarrow{\Sigma^{\ell}\bot}\Sigma^{\ell}Y
\quad\text{and}\quad
\Sigma^{\ell}\top\colon\Delta[\ell]\to\Sigma^{\ell}\Delta[0]
\xrightarrow{\Sigma^{\ell}\top}\Sigma^{\ell}Y.\]
corresponding to the two $0$-simplices $\bot$ and $\top$ of $Y$;
they are not marked because $\bot$ and $\top$ weren't.
\item[($B_\ell(Y)''$)]
For each non-degenerate $(d-1)$-simplex $\tau'$ of $X$, the $(d+\ell)$-simplex
\[\Sigma^\ell(\Sigma \tau')\colon 
\Delta[d+\ell]\to\Sigma^{\ell}\Delta[d]\to \Sigma^{\ell}\Sigma \Delta[d-1]\xrightarrow{\Sigma^{\ell}\Sigma\tau'}=\Sigma^{\ell}\Sigma X\]
which is marked if and only if $\Sigma\tau'$ is marked if and only if $\tau'$ is marked.
\end{itemize}
This gives the desired final list; specifically with
\[
A_{\ell+1}(X)=A_\ell(Y) \cup B_\ell(Y)'
\quad\text{and}\quad
B_{\ell+1}(X)=B_\ell(Y)''.
\]
These are exactly the simplices listed in the statement for $\ell+1$, and the marking condition is the same one. Hence the claim follows.
\end{proof}


In particular, we record which non-degenerate simplices remain non-degenerate
after applying the quotient map from
\cref{ProgressiveSuspensionsGeneral}:

\begin{rmk}
\label{SurvivingNonDegenerate}
Let $d\geq 1$ and $p\geq 1$.
Given a marked simplicial set $X$, a non-degenerate $d$-simplex
$\rho=\sigma\star\tau$ of $X\star\Delta[p-1]$
remains non-degenerate in $\Sigma^pX$
under the quotient map $X\star\Delta[p-1]\to \Sigma^p X$
of \cref{ProgressiveSuspensionsGeneral}
if and only if one of the following holds:
\begin{itemize}
    \item
    $\dim(\sigma)\geq 0$ and $\tau=\{0,\dots,p-1\}$, or
    \item
    $\dim(\sigma)\leq 0$ and $\tau$ is a terminal segment of $\{0,\dots,p-1\}$, or
    \item
    $\sigma=\varnothing$
    and $\tau=\{i\}\star \{j,\dots, p-1\}$ for some $j>i+1$.
\end{itemize}
In particular, if $\dim \rho \geq p$, then $\tau=\{0,\dots, p-1\}$.
The same statement is also true for $p=0$, vacuously.
\end{rmk}

We will later make use of a different filtration for the map $X \to 
\Sigma^{\ell+1}X$:

\begin{rmk}
\label{ProgressiveSuspensionsGeneral2}
The composite $\alpha_\ell\colon X \star \Delta[\ell] \to \Sigma^{\ell+1} X$ from \cref{ProgressiveSuspensionsGeneral} can also be defined inductively by the formula
\[
\alpha_\ell\colon
X \star \Delta[\ell]=(X \star \Delta[\ell-1]) \star \Delta[0]\to
\Sigma (X \star \Delta [\ell -1])
\xrightarrow{\Sigma (\alpha_{\ell-1})}\Sigma (\Sigma^{\ell} X)=\Sigma^{\ell+1}X
\]
with $\alpha_{-1}=\id_X\colon X\to X$.
This yields a filtration
\[
X \star \Delta[\ell]
\to
\Sigma( X \star \Delta[\ell -1])
\to\dots\to
\Sigma^{p}(X\star \Delta[\ell-p])
\to\dots\to
\Sigma^{\ell}(X\star\Delta[0])
\to
\Sigma^{\ell+1}X.
\]

The fact that both filtrations give rise to the same resulting composite
$X\star\Delta[\ell]\to \Sigma^{\ell+1}$ follows from the naturality in $X$ of the quotient map
$q_X\colon X\star \Delta[0] \to \Sigma X$. For example, in the case $\ell = 1$ we have the following naturality square:
\[
\begin{tikzcd}
{X \star [0]\star [0]}\ar[r,"{q_{X\star[0]}}"]\ar[d, "{q_X \star [0]}"swap]&{\Sigma(X\star [0])}\arrow[d,"{\Sigma(q_X)}"]\\
{\Sigma X \star [0]}\arrow[r,"{q_{\Sigma X}}"]&\Sigma \Sigma X
\end{tikzcd}
\]
and the case of general $\ell\geq1$ is analogous.
\end{rmk}

\subsection{Roberts--Street nerve preserves homs}

Marked simplicial sets also serve as an ambient for $\omega$-categories, through various constructions:

\begin{notn}
Let $\cD$ be an $\omega$-category.
\begin{itemize}[leftmargin=*]
    \item The simplicial set $N\cD$ is the \emph{nerve} of $\cD$ from \cite{StreetOrientedSimplexes}.
    Explicitly, the set of $k$-simplices is given by
    \[(N\cD)_k=\omega\cat(\cO[k],\cD)\]
    The nerve construction defines a continuous functor $N\colon \omega\cat\to\sset$.
    \item The marked simplicial set $N^{\mathrm{RS}}\cC$ is the \emph{Roberts--Street nerve} of $\cC$ from \cite[Observation 246]{VerityComplicialAMS}. This is given by endowing $N\cC$ with the marking of all simplices whose corresponding top dimensional cell is an identity. Explicitly, the set of (marked) $k$-simplices is given by
    \[(N^{\mathrm{RS}}\cD)_{k}=\omega\cat(\cO[k],\cD)\quad\text{ and }\quad t(N^{\mathrm{RS}}\cD)_{k}=\omega\cat(\cO[k]\amalg_{\cC[k]}\cC[k-1],\cD),\]
    following the convention that $\cC[-1]=\varnothing$.
    The Roberts--Street nerve construction defines a continuous functor $N^{\mathrm{RS}}\colon \omega\cat\to m\sset$.
    \item The marked simplicial set $N^\natural\cC$ is the \emph{natural nerve} of $\cC$ from \cite[Définition\ 5.10]{Loubaton1}. This is given by endowing $N\cC$ with the marking of all simplices whose corresponding top dimensional cell is coinductively invertible.
    Explicitly, if $\cE$ denotes the $\omega$-category from \cref{E:exists}, the set of marked $k$-simplices is given by the simplices represented by a map $\cO[k]\to\cD$ which extend to a map $\cO[k]\amalg_{\Sigma^{k-1}\cC[1]}\Sigma^{k-1}\cE\to\cD$.
The natural nerve construction defines a functor $N^{\natural}\colon \omega\cat\to m\sset$.
\end{itemize}
\end{notn}

\begin{rmk}
\label{NerveAdjunction}
As considered in \cite[\textsection5]{StreetOrientedSimplexes},
there is an adjunction
\[c\colon \sset\rightleftarrows\omega\cat\noloc N.\]
Here, the left adjoint $c\colon \sset\to\omega\cat$ acts
on
representable simplicial sets as:
\[c\Delta[k]\cong\cO[k]\text{ for $k\geq0$}.\]
\end{rmk}

\begin{rmk}
\label{NerveAdjunctionMarked}
We know from \cite[Observation 246]{VerityComplicialAMS} that there is an adjunction
\[
c^{\mathrm{RS}}\colon m\sset\rightleftarrows\omega\cat\noloc N^{\mathrm{RS}}.
\]
Here, the left adjoint $c^{\mathrm{RS}}\colon m\sset\to\omega\cat$ acts on (marked) representable simplicial sets as:
\[c^{\mathrm{RS}}\Delta[k]\cong\cO[k]\text{ for $k\geq0$}
\quad{\text{ and }}\quad
c^{\mathrm{RS}}\Delta[k]_t\cong\cO[k]\amalg_{\cC[k]}\cC[k-1]\text{ for $k>0$}.\]
\end{rmk}

The goal of this subsection is to prove  
\cref{NvsHom}, which establishes the compatibility between the Roberts--Street nerve $N^{\mathrm{RS}}$ and the hom functor(s). This fact will be employed in \cref{thNerveContractibleGaunt}.

\begin{rmk}
\label{IsoOnReps}
Denote by $t\Delta$ the full subcategory of $m\sset$ spanned by all marked simplicial sets of the form $\Delta[\ell]$ (standard $\ell$-simplex with minimal marking) for $\ell\geq0$ and $\Delta[\ell]_t$ (standard $\ell$-simplex with minimal marking that includes the top simplex) for $\ell>0$.
As discussed in \cite[Observation 108]{VerityComplicialAMS} and \cite[\textsection 1]{or},
$m\sset$ is a reflective subcategory of $\set^{t\Delta^{\mathrm{op}}}$ containing all representables.
Then, for all cocomplete categories $\cC$,
if $\mathrm{Fun}^{\colim}(m\sset,\cC)$ resp.~($ \mathrm{Fun}^{\colim}(\set^{t\Delta^{\mathrm{op}}},\cC)$) denotes the category of colimit-preserving functors from $m\sset$ (resp.~$\set^{t\Delta^{\mathrm{op}}}$) to $\cC$, the functor
\[\mathrm{Refl}^*\colon\mathrm{Fun}^{\colim}(msSet,\cC)\to \mathrm{Fun}^{\colim}(\set^{t\Delta^{\mathrm{op}}},D)\cong \mathrm{Fun}(t\Delta,\cC)\]
induced by the reflector $\mathrm{Refl}\colon\set^{t\Delta^{\mathrm{op}}}\to m\sset$ is fully faithful.
\end{rmk}

\begin{lem}
\label{CvsCone}
For all
$\ell\geq0$ (resp.~$\ell>0$)
there are natural isomorphisms (with respect to morphisms in $t\Delta$) in $\omega\cat$
\[
c^{\mathrm{RS}}(\Delta[\ell]_{(t)}\star\Delta[0])\cong c^{\mathrm{RS}}\Delta[\ell]_{(t)}\star [0].
\]
\end{lem}

\begin{proof}
The isomorphism
\[ c^{\mathrm{RS}}(\Delta[\ell]\star\Delta[0])\cong c^{\mathrm{RS}}\Delta[\ell]\star [0]\]
for $\ell\geq0$ is by construction, and we now treat the isomorphism
\[
c^{\mathrm{RS}}(\Delta[\ell]_t\star\Delta[0])\cong c^{\mathrm{RS}}\Delta[\ell]_t\star [0]\]
for $\ell>0$. For this, we see using the definition of join from \cite[\textsection3]{VerityComplicialI} that the marked simplicial set 
$\Delta[\ell]_t\star \Delta[0]$ can be described as a pushout in $m\sset$ of the form
\[
\begin{tikzcd}
    \coprod\limits_{a=-1}^0\Delta[\ell] \star \Delta[a]\arrow[d]\arrow[r]& \Delta[\ell] \star \Delta[0]\arrow[d] \\
    \coprod\limits_{a=-1}^0\Delta[\ell+1+a]_t\arrow[r]& \Delta[\ell]_t \star \Delta [0]
\end{tikzcd}
\]
Thus, by applying the cocontinuous functor $c^{\mathrm{RS}}$ we obtain a pushout in $\omega\cat$ of the form
\[
\begin{tikzcd}
    \coprod\limits_{a=-1}^0\cO[\ell+1+a]\arrow[d]\arrow[r]&\cO[\ell+1] \arrow[d]\\
    \coprod\limits_{a=-1}^0\cO[\ell+1+a]\amalg_{\cC[\ell+1+a]}\cC[\ell+1+a-1]\arrow[r]&c^{\mathrm{RS}}(\Delta[\ell]_t\star \Delta[0])
\end{tikzcd}
\]
By pushout pasting,
we obtain that $c^{\mathrm{RS}}( \Delta[\ell]_t\star \Delta[0])$ 
can be described as a pushout in $\omega\cat$ of the form
\[
\begin{tikzcd}    \coprod\limits_{a=-1}^0\cC[\ell+1+a]\arrow[d]\arrow[r]&\cO[\ell+1] \arrow[d]\\
    \coprod\limits_{a=-1}^0\cC[\ell+1+a-1]\arrow[r]&c^{\mathrm{RS}}(\Delta[\ell]_t\star\Delta[0])
\end{tikzcd}
\]
Furthermore, the $\omega$-category $(c^{\mathrm{RS}}\Delta[\ell]_t)\star [0]$ can be written as
\[
\begin{array}{lll}
(c^{\mathrm{RS}}\Delta[\ell]_t) \star [0] &\cong  (\cO[\ell]\underset{\cC[\ell]}{\amalg}\cC[\ell-1])\star [0]\\
& \cong (\cO[\ell]\star[0]) \underset{\cC[\ell]\star[0]}{\amalg}(\cC[\ell-1]\star[0])\\
& \cong \cO[\ell+1] \underset{\cC[\ell]\star[0]}{\amalg}(\cC[\ell-1]\star[0]).
\end{array}
\]
In other words, we have pushout square in $\omega\cat$:
\[
\begin{tikzcd}
    \cC[\ell]\star[0]\arrow[r]\arrow[d]&\cO[\ell+1]\arrow[d]\\
    \cC[\ell-1]\star[0]\arrow[r]&({c^{\mathrm{RS}}\Delta[\ell]_t)\star [0]}
\end{tikzcd}
\]
By \cref{ConeCell3} 
we also have that
there is a pushout in $\omega\cat$
\[
\begin{tikzcd}
\coprod\limits_{a=-1}^0\cC[\ell+1+a]\arrow[d]\arrow[r]&\cC[\ell] \star [0]\arrow[d]\\
\coprod\limits_{a=-1}^0\cC[\ell+a]\arrow[r]&\cC[\ell-1]\star[0]
\end{tikzcd}
\]
By pushout pasting, we obtain that $(c^{\mathrm{RS}}\Delta[\ell]_t) \star [0]$ can be described as a pushout in $\omega\cat$
\[
\begin{tikzcd}    \coprod\limits_{a=-1}^0\cC[\ell+1+a]\arrow[d]\arrow[r]&\cO[\ell+1] \arrow[d]\\
    \coprod\limits_{a=-1}^0\cC[\ell+1+a-1]\arrow[r]&({c^{\mathrm{RS}}\Delta[\ell]_t)\star [0]}
\end{tikzcd}
\]
Hence, the claim follows.
By construction, the isomorphism is natural with respect to the canonical map $\Delta[\ell]\to\Delta[\ell]_t$. One can check that it also is natural with respect to the maps $\Delta[\ell]_t\to\Delta[\ell-1]$. From this, given the description of $t\Delta$ in terms of generators and relations from \cite[Notation 1.1]{or}, one deduces the naturality with respect to all maps in $t\Delta$.
\end{proof}

\begin{prop}
\label{CvsSigma}
For any marked simplicial set $X$,
there is a natural isomorphism of bipointed $\omega$-categories
\[
  c^{\mathrm{RS}}\Sigma X\cong \Sigma c^{\mathrm{RS}}X.
\]
\end{prop}

\begin{proof}
Given the adjunctions from \cref{SuspensionAdjunctionCat,SuspensionAdjunctionMarked,NerveAdjunctionMarked}, both sides are cocontinuous in $X$.
Hence, by \cref{IsoOnReps}, it suffices to prove it for representables.
Now, for $\ell\geq0$ (resp.~$\ell>0$), we have natural isomorphisms of bipointed $\omega$-categories
\[
\begin{array}{lll}
c^{\mathrm{RS}}\Sigma\Delta[\ell]_{(t)}&\cong c^{\mathrm{RS}}(\Delta[\ell]_{(t)}\star\Delta[0]/\Delta[\ell]_{(t)})&\text{\cref{SuspensionMarked}}\\
&\cong c^{\mathrm{RS}}(\Delta[\ell]_{(t)}\star\Delta[0])/c^{\mathrm{RS}}\Delta[\ell]_{(t)}&\text{\cref{NerveAdjunctionMarked}}\\
&\cong(c^{\mathrm{RS}}(\Delta[\ell]_{(t)})\star [0])/c^{\mathrm{RS}}\Delta[\ell]_{(t)}&\text{\cref{CvsCone}}\\
&\cong\Sigma c^{\mathrm{RS}}\Delta[\ell]_{(t)},&\text{\cref{SuspensionCat}}
\end{array}
\]
concluding the proof.
\end{proof}

\begin{thm}
\label{NvsHom}
For any $\omega$-category $\cC$
and objects $x$ and $y$, there is a natural isomorphism of marked simplicial sets
\[
(N^{\mathrm{RS}}\cC)(x,y)\cong N^{\mathrm{RS}}(\cC(x,y)).
\]  
\end{thm}

\begin{proof}
This follows from \cref{CvsSigma,SuspensionAdjunctionMarked,SuspensionAdjunctionCat,NerveAdjunctionMarked}.
\end{proof}

We now discuss the compatibility of $c^{\mathrm{RS}}$ with the iterated suspension functor:

\begin{prop}
\label{CvsIteratedSigma}
Let $\ell>0$. Given any marked simplicial set $X$ there is a natural isomorphism of $\omega$-categories under $\partial\cC[\ell]$
\[
c^{\mathrm{RS}}\Sigma^\ell X\cong\Sigma^\ell c^{\mathrm{RS}}X.
\]
\end{prop}

\begin{proof}
Given the adjunctions from \cref{SuspensionAdjunctionCat,SuspensionAdjunctionMarked,NerveAdjunctionMarked}, both sides are cocontinuous in $X$.
Hence, by \cref{IsoOnReps}, it suffices to prove it for representables.
Now, for $\ell\geq0$ (resp.~$\ell>0$), by \cref{CvsCone,CvsSigma} we have natural isomorphisms of $\omega$-categories under $\partial\cC[\ell]$
\[
\begin{array}{lll}
  c^{\mathrm{RS}}(\Sigma^{\ell+1} \Delta[\ell]_{(t)})
  &\cong c^{\mathrm{RS}}((\Sigma^\ell X\star\Delta[0])/\Sigma^\ell \Delta[\ell]_{(t)})\\
  &\cong(c^{\mathrm{RS}}\Sigma^\ell \Delta[\ell]_{(t)}\star\cO[0])/c^{\mathrm{RS}}\Sigma^\ell \Delta[\ell]_{(t)}\\
  &\cong(\Sigma^\ell c^{\mathrm{RS}}\Delta[\ell]_{(t)}\star\cO[0])/\Sigma^\ell c^{\mathrm{RS}}\Delta[\ell]_{(t)}\\
  &\cong\Sigma(\Sigma^{\ell}c^{\mathrm{RS}}\Delta[\ell]_{(t)})\\
  &\cong
  \Sigma^{\ell+1} c^{\mathrm{RS}}\Delta[\ell]_{(t)},
\end{array}
\]
concluding the proof.
\end{proof}

\begin{prop}
\label{NvsIteratedHom}
Given $(\cC,\delta\colon \partial\cC[\ell]\to \cC)$ in ${}^{\partial\cC[\ell]/}\omega\cat$,
there is a natural isomorphism of marked simplicial sets
\[
N^{\mathrm{RS}}(\cC^{(\ell)}(\delta))\cong (N^{\mathrm{RS}}\cC)^{(\ell)}(\widehat\delta),
\]
where $\widehat\delta\colon \Sigma^{\ell}\varnothing\to N^{\mathrm{RS}}\cC$ is the transpose of $\delta\colon \partial\cC[\ell]\to \cC$ through the adjunction from \cref{NerveAdjunctionMarked}.
\end{prop}

\begin{proof}
Using \cref{CvsIteratedSigma}, we know that for $\ell\ge0$  there is an isomorphism of $\omega$-categories
\[
c^{\mathrm{RS}}\Sigma^{\ell}\varnothing\cong\partial\cC[\ell],
\]
and that there is a commutative diagram of left adjoint functors
\[
\begin{tikzcd}
m\sset\arrow[r,"\Sigma^{\ell}"]\arrow[d,"c^{\mathrm{RS}}"swap]&{}^{\Sigma^{\ell}\varnothing/}m\sset
\arrow[d,"c^{\mathrm{RS}}"]\\
\omega\cat\arrow[r,"\Sigma^{\ell}"swap]&{}^{\partial\cC[\ell]/}\omega\cat
\end{tikzcd}
\]
By \cref{IteratedSuspensionAdjunctionMarked,IteratedSuspensionAdjunctionCat}, we then obtain a diagram of right adjoint functors
\[
\begin{tikzcd}
m\sset&{}^{\Sigma^{\ell}\varnothing/}m\sset\arrow[l,"(-)^{(\ell)}(-)"swap]\\
\omega\cat\arrow[u,"N^{\mathrm{RS}}"]&{}^{\partial\cC[\ell]/}\omega\cat\arrow[u,"\widehat{(-)}"swap]\arrow[l,"(-)^{(\ell)}(-)"]
\end{tikzcd}
\]
The statement follows.
\end{proof}


\section{About complicial sets - homotopy theory}

\subsection{Background on complicial sets}

\label{ComplicialSets}

In this section we will assume the reader to be familiar with the basics of model category theory; see e.g.~\cite{hovey}.

Now introduce complicial sets and homotopy theory. The notation and terminology from this subsection is standard (cf.~\cite{VerityComplicialAMS,VerityComplicialI,EmilyNotes,or,RiehlVerityBook}).

\begin{notn}
\label{ComplicialNotation}
We denote
\begin{itemize}[leftmargin=*]
    \item by $\Delta^k[m]$, for $0\leq k \leq m$, the simplicial set given by the standard $m$-simplex in which a non-degenerate simplex is marked if and only if it contains the vertices $\{k-1,k,k+1\}\cap [m]$;
    \item by $\Delta^k[m]'$, for $0\leq k \leq m$, the simplicial set given by the standard $m$-simplex with marking obtained from $\Delta^k[m]$ by additionally marking the $(k-1)$-st and $(k+1)$-st $(m-1)$-dimensional face of $\Delta[m]$, whenever defined;
    \item by $\Delta^k[m]''$, for $0\leq k \leq m$, the simplicial set given by the standard $m$-simplex with marking obtained from $\Delta^k[m]'$ by additionally marking the $k$-th face of $\Delta[m]$;
    \item by $\Lambda^k[m]$, for $0\leq k \leq m$, the simplicial set given by the usual $k$-horn $\Lambda^k[m]$ with marking inherited from $\Delta^k[m]$.
    \item by $\Delta[m]^{\sharp}$ the simplicial set given by the standard $m$-simplex with the maximal marking.
    \item by $\Delta[m]_{t}$ the standard $m$-simplex in which the top $m$-dimensional simplex is marked, as well as all degenerate simplices.
    \item by $\Delta[3]^{\mathrm{eq}}$ the standard $3$-simplex in which the $1$-simplices $[0,2]$ and $[1,3]$ are marked, as well as all degenerate $1$-simplices and all simplices in dimension $2$ or higher.
    \item by $\Delta[3]^{\mathrm{eq}}\star\Delta[\ell]$ the simplicial set given by the join of a standard $3$-simplex with a standard $\ell$-simplex, which is isomorphic to the standard $(3+1+\ell)$-simplex, in which a simplex $\sigma\star\tau$ is marked if and only if $\sigma$ is marked in $\Delta[3]^{\mathrm{eq}}$ or $\tau$ is a degenerate simplex of $\Delta[\ell]$.
    \item by $\Delta[3]^{\sharp}\star\Delta[\ell]$ the simplicial set given by the join of a standard $3$-simplex with a standard $\ell$-simplex, which is isomorphic to the standard $(3+1+\ell)$-simplex, in which a simplex $\sigma\star\tau$ is marked if and only if $\sigma$ is marked in $\Delta[3]^{\sharp}$ (that is, $\sigma$ has positive dimension) or $\tau$ is a degenerate simplex of $\Delta[\ell]$.
\end{itemize}
\end{notn}

\begin{defn}
\label{ComplicialAnodyne}
Let $X$ be a marked simplicial set.
\begin{itemize}[leftmargin=*]
    \item We say that $X$ is a \emph{complicial set} if it has the right lifting property with respect to the \emph{complicial horn extension}
    \[\Lambda^k[m]\to \Delta^k[m].\]
    for $m> 1$ and $0\leq k\leq m$
    and with respect to the \emph{complicial thinness extension}
    \[\Delta^k[m]' \to \Delta^k[m]''\]
    for $m\geq 2$ and $0\leq k \leq m$.
    \item We say that $X$ is a \emph{saturated complicial set} if it is a complicial set and it has the right lifting property with respect to
    \emph{complicial saturation extension}
    \[
    \Delta[3]^{\mathrm{eq}}\star\Delta[m]\to\Delta[3]^{\sharp}\star\Delta[m]\]
    for $m\ge-1$.
    \item We say that $X$ is an \emph{$n$-complicial set} if it is a saturated complicial set and it has the right lifting property with respect to
    \emph{triviality extension}
\[\Delta[m]\to\Delta[m]_t.\]
for $m>n$.
\end{itemize}
\end{defn}

We consider the following model structures on marked simplicial sets:

\begin{notn}
\label{ModelStructuresMarked}
We consider the following model structures on $m\sset$:
\begin{itemize}[leftmargin=*]
    \item Let $m\sset_{\mathrm{cmp}}$ denote the model structure on $m\sset$ from \cite[\textsection6]{VerityComplicialI}, in which the fibrant objects are the \emph{complicial sets} (there referred to as ``weak complicial sets'') and the cofibrations are the monomorphisms on underlying simplicial sets.
    \item Let $m\sset_{\mathrm{cmp,sat}}$ denote the model structure on $m\sset$ from \cite[Example 3.3.5]{EmilyNotes} (established in \cite[\textsection1.3]{or}), in which the fibrant objects are the \emph{saturated complicial sets} and the cofibrations are the monomorphisms. This model structure is obtained from $m\sset_{\mathrm{cmp}}$ by localizing at the saturation anodyne extensions for $\ell\geq-1$
    \[\Delta[3]^{\mathrm{eq}}\star\Delta[\ell]\hookrightarrow
    \Delta[3]^\sharp\star\Delta[\ell].\]
    \item Let $m\sset_{(\infty,n)}$ denote the model structure on $m\sset$ from \cite[\textsection1.3]{or}, in which the fibrant objects are the $n$-trivial saturated complicial sets, a.k.a.\ the \emph{$n$-complicial sets} and the cofibrations are the monomorphisms. This model structure is obtained from $m\sset_{\mathrm{cmp,sat}}$ by localizing at the triviality extensions for $p>n$
    \[\Delta[p]\to\Delta[p]_t.\]
    \end{itemize}   
\end{notn}

\begin{rmk}
By \cref{ThOfEntire}, if $A \to A'$ is an acyclic cofibration in one of the model structure on $m\sset$ from \cref{ModelStructuresMarked} and is an isomorphism on underlying simplicial sets, then $\mathrm{th}_n$ takes it to an acyclic cofibration in the same model structure.
\end{rmk}

We consider the following Quillen pairs:

\begin{prop}
\label{TruncationsQuillenMarked}
For all $n\geq0$, the functors from \cref{TruncationsAdjunctionMarked} define Quillen pairs:
\begin{enumerate}[ref=(\arabic*)]
    \item\label{QPcmp} $\mathrm{th}_n\colon m\sset_{\mathrm{cmp}}\rightleftarrows m\sset_{\mathrm{cmp}}\noloc\mathrm{sp}_n$;
\item \label{QPsat}$\mathrm{th}_n\colon m\sset_{\mathrm{cmp,sat}}\rightleftarrows m\sset_{\mathrm{cmp,sat}}\noloc\mathrm{sp}_n$;
\item\label{QPnleft} $
\mathrm{th}_n\colon m\sset_{(\infty,n)}\rightleftarrows m\sset_{\mathrm{cmp,sat}}\noloc\mathrm{sp}_n
$;
\item \label{QPnright}$
\mathrm{th}_n\colon m\sset_{\mathrm{cmp,sat}}\rightleftarrows m\sset_{(\infty,n)}\noloc\mathrm{sp}_n$.
\end{enumerate}
\end{prop}

\begin{proof}
  It was shown in \cite[Lemma 25]{VerityComplicialI} that $\mathrm{th}_n$ sends any complicial horn extension and any thinness extension to an acyclic cofibration in $m\sset_{\textrm{cmp}}$.
  So, using a variant of \cite[Lemma 1.8]{ORfp} (which builds on \cite[Proposition 7.15]{JT}), we obtain the Quillen pair \ref{QPcmp}.

Using \cref{ThOfEntire}, we see that the functor $\mathrm{th}_n$ sends any saturation extension (which is an entire map) to a pushout of a saturation extension. So, using \cite[Theorem 3.3.20]{hirschhorn} on the Quillen pair from \ref{QPcmp},
we obtain the Quillen pair \ref{QPsat}.

By inspection we see that $\mathrm{th}_n$ sends triviality maps in dimensions higher than $n$ to isomorphisms. So, using \cite[Theorem 3.3.20]{hirschhorn} on the Quillen pair \ref{QPsat}, we obtain the Quillen pair \ref{QPnleft}.

Composing the Quillen pair from \ref{QPsat}
with the localization
\[
\id\colon m\sset_{\mathrm{cmp,sat}}\rightleftarrows m\sset_{(\infty,n)}\noloc\id
\]
we obtain the Quillen pair \ref{QPnright}.
\end{proof}

\begin{prop}
\label{SuspensionQuillenMarked}
The suspension-hom adjunction from \cref{IteratedSuspensionAdjunctionMarked} forms a Quillen adjunction
\[
\Sigma\colon m\sset_{\mathrm{cmp,sat}}
\rightleftarrows
{}^{\partial\Delta[1]/}m\sset_{\mathrm{cmp,sat}}
\noloc(-)^{(-)}(-),
\]
where the undercategory ${}^{\partial\Delta[1]/}m\sset_{\mathrm{cmp,sat}}
$ is equipped with the induced model structure of \cite{HirschhornOvercategories}.
\end{prop}

\begin{proof}
The proof of \cite[Lemma 2.5]{ORfp} can be adjusted to conclude that the cone functor
\[
(-)\star\Delta[0]\colon m\sset_{\mathrm{cmp,sat}}
\to
{}^{\Delta[0]/}m\sset_{\mathrm{cmp,sat}}
\]
is left Quillen. Using this, the argument of \cite[Lemma 2.7]{ORfp} can be adjusted to conclude that the suspension functor
\[
\Sigma\colon m\sset_{\mathrm{cmp,sat}}
\to
{}^{\partial\Delta[1]/}m\sset_{\mathrm{cmp,sat}},
\]
is left Quillen, as desired.
\end{proof}

\begin{prop}
\label{IteratedQuillenAdjunctionMarked}
For $\ell>0$, the Quillen adjunction from \cref{SuspensionQuillenMarked} induces a Quillen adjunction
\[\Sigma^{\ell}\colon m\sset_{\mathrm{cmp,sat}}\rightleftarrows{}^{\Sigma^{\ell}\varnothing/}m\sset_{\mathrm{cmp,sat}}\noloc(-)^{(\ell)}(-)\]
where ${}^{\Sigma^{\ell}\varnothing/}m\sset_{\mathrm{cmp,sat}}$ denotes the model category of simplicial sets under $\Sigma^{\ell}\varnothing$ from \cite{HirschhornOvercategories}.
\end{prop}

\begin{proof}
This is an instance of \cref{IteratedQuillenAdjunction}, applied to
$\cC=m\sset_{\mathrm{cmp,sat}}$, $c_0=\varnothing$, $F=\Sigma$, and $G=(-)^{(-)}(-)$,
using the suspension-hom Quillen adjunction of \cref{SuspensionQuillenMarked}.
\end{proof}

\subsection{Properties of nerves of contractible gaunt categories}

\label{NerveOfContractibleGaunt}

We establish as \cref{NEnotContractible,NEfibrant} two properties of the Roberts--Street nerve of gaunt $\omega$-categories.

\begin{rmk}
\label{RmkNaturalRS}
If $\cD$ is a gaunt $k$-category for some $k\geq0$, then
$N^{\mathrm{RS}}\cD=N^\natural\cD$.
\end{rmk}

\begin{prop}
\label{NEfibrant}
If $\cG$ is a 
gaunt $\omega$-category, the unique map of marked simplicial sets $N^{\mathrm{RS}}\cG\to\Delta[0]$ is a fibration in $m\sset_{\mathrm{cmp,sat}}$.
\end{prop}

\begin{proof}
The fact that $N^{\mathrm{RS}}\cG\to\Delta[0]$ is a fibration in $m\sset_{\mathrm{cmp}}$ follows from \cite[Theorem~266]{VerityComplicialAMS}, and we now show that it is a fibration in $m\sset_{\mathrm{cmp,sat}}$ by showing it has the right lifting property with respect to the saturation extension
\[
\Delta[3]^\mathrm{eq}\star\Delta[\ell]\to\Delta[3]^\sharp\star\Delta[\ell]
\]
for all $\ell\geq-1$, concluding the proof.

To this end, consider a map of marked simplicial sets
\[\Delta[3]^\mathrm{eq}\star\Delta[\ell]\to N^{\mathrm{RS}}\cG.\]
For dimension reasons, it needs to factor through
\[\Delta[3]^\mathrm{eq}\star\Delta[\ell]\to \mathrm{sk}_{3+1+\ell}N^{\mathrm{RS}}\cG,\]
where $\mathrm{sk}_{3+1+\ell}N^{\mathrm{RS}}\cG$ denotes the usual simplicial $(3+1+\ell)$-skeleton with the marking inherited from $N^{\mathrm{RS}}\cG$. By definition of core in \cref{TruncationsCat} and dimension reasons, it needs to factor through
\[\Delta[3]^\mathrm{eq}\star\Delta[\ell]\to\mathrm{sk}_{3+1+\ell}N^{\mathrm{RS}}\mathrm{core}_{3+1+\ell}\cG.\]
Since $\cG$ is a gaunt $\omega$-category, its core $\mathrm{core}_{3+1+\ell}\cG$ is a gaunt
$(3+1+\ell)$-category. In particular, using \cref{RmkNaturalRS}, it factors through
\[
\Delta[3]^\mathrm{eq}\star\Delta[\ell]\to N^{\mathrm{RS}}\mathrm{core}_{3+1+\ell}\cG=N^{\natural}\mathrm{core}_{3+1+\ell}\cG.\]
By \cite[Théorème 5.22]{Loubaton1},
the marked simplicial set $N^{\natural}\mathrm{core}_{3+1+\ell}\cG$ is saturated, so using again \cref{RmkNaturalRS} we see that the map of marked simplicial sets can be lifted to
\[
\Delta[3]^\sharp\star\Delta[\ell]\to N^{\natural}\mathrm{core}_{3+1+\ell}\cG=N^{\mathrm{RS}}\mathrm{core}_{3+1+\ell}\cG.\]
By postcomposing appropriately, we then obtain a map of marked simplicial sets
\[
\Delta[3]^\sharp\star\Delta[\ell]\to N^{\mathrm{RS}}\cG,\]
as desired.
\end{proof}

\begin{prop}
\label{NEnotContractible}
Given a gaunt $\omega$-category $\cG$, if the map of marked simplicial sets $N^{\mathrm{RS}}\cG\to\Delta[0]$ is a weak equivalence
in $m\sset_{\mathrm{cmp,sat}}$,
then $\cG$ is trivial.
\end{prop}

\begin{proof}
First, the unique map of marked simplicial sets $f\colon N^{\mathrm{RS}}\cG\to\Delta[0]$ is a weak equivalence in $m\set_{\textrm{cmp,sat}}$ by assumption and a fibration by \cref{NEfibrant}. Hence, it is a trivial fibration in $m\set_{\textrm{cmp,sat}}$, and the identity map defines a cofibration $\id_{N\cG}\colon N^{\mathrm{RS}}\cG\to\mathrm{th}_0N\cG$.
Hence, there is a solution to the lifting problem in $m\sset$:
\[
\begin{tikzcd}
N^{\mathrm{RS}}\cG\arrow[r,"\id_{N\cG}"]\arrow[d,"\id_{N\cG}"swap,hook]&N^{\mathrm{RS}}\cG\arrow[d,"\simeq", two heads]\\
\mathrm{th}_0N\cG\arrow[r]\arrow[ru,dashed]&\Delta[0]   
\end{tikzcd}
\]
It follows that every simplex in $N^{\mathrm{RS}}\cG$ is marked, meaning that every morphism in $\cG$ is an identity.

Next, if $x$ and $y$ are objects of $\cG$,
there is a solution to the lifting problem in $m\sset$:
\[
\begin{tikzcd}
\partial\Delta[1]\arrow[r,"{(x,y)}"]\arrow[d,hook]&N^{\mathrm{RS}}\cG\arrow[d,"\simeq", two heads]\\
\Delta[1]\arrow[r]\arrow[ru,dashed,"h"]&\Delta[0]
\end{tikzcd}
\]
It follows that there is a morphism $h\colon x\to y$, 
meaning that $x=y$, so $\cG$ has at most one object.

Finally, there is a solution to the lifting problem in $m\sset$:
\[
\begin{tikzcd}
\varnothing\arrow[r]\arrow[d,hook]&N^{\mathrm{RS}}\cG\arrow[d,"\simeq", two heads]\\
\Delta[0]\arrow[r]\arrow[ru,dashed]&\Delta[0]
\end{tikzcd}
\]
so $\cG$ has at least one object.

In total,
we showed that $\cG$ has a unique object and
a unique morphism in each positive dimension.
Hence, $\cG$ is a trivial $\omega$-category, as desired.
\end{proof}

\subsection{Thinifications of Roberts--Street nerve of gaunt contractible categories}

\label{thNerveContractibleGaunt}

The goal is to prove \cref{LnEcontractible}, which establishes the contractibility of $\mathrm{th}_n N^{\mathrm{RS}}\cE$ in the homotopy theory of $(\infty,n)$-categories.

\begin{lem}
\label{NPcontractibleKan}
For any contractible $\omega$-category $\cP$, the nerve $N\cP$ is a contractible Kan complex.
\end{lem}

\begin{proof}
A lifting problem of the form
\[
\begin{tikzcd}
\partial\Delta[n]\arrow[r]\arrow[d]&N\cP\\
\Delta[n]\arrow[ru,dashed, "?"swap]
\end{tikzcd}
\]
corresponds to a lifting problem
\[
  \begin{tikzcd}
    \partial\cO[n]\arrow[r]\arrow[d]&\cP\\
    \cO[n]\arrow[ru,dashed, "?"swap]
  \end{tikzcd}
\]
Since $\partial\cO[n]\subseteq\cO[n]$ is a relative polygraph
obtained by attaching a single $n$-cell,
the solution of such a lifting problem
amounts to specifying a single $n$-cell in $\cP$
with prescribed $(n-1)$-dimensional source and boundary.
This is always possible by \cref{E:Parallel}.
\end{proof}

\begin{prop}
\label{InductionBase}
For any contractible $\omega$-category $\cP$, the canonical map
\[
\mathrm{th}_1N^{\mathrm{RS}}\cP\hookrightarrow \mathrm{th}_0N^{\mathrm{RS}}\cP=N^\sharp\cP
\]
is an acyclic cofibration in $m\sset_{\mathrm{cmp,sat}}$.
\end{prop}

\begin{proof}
Given any $1$-simplex $\alpha$ of $N\cP$ represented by a map $\widetilde\alpha\colon\Delta[1]\to N\cP$, we construct a map
$\widetilde\alpha\colon\Delta[3]\to N\cP$
as follows.

If $\Delta[0]\amalg_{\Delta[1]}\Delta[3]\amalg_{\Delta[1]}\Delta[0]$ denotes the $3$-simplex with degenerate diagonals, there exist extensions $\widetilde\alpha$ and $\overline\alpha$ of $\alpha$ which fit into the following commutative diagram:
 \[
\begin{tikzcd}
\Delta[1]\arrow[rr,"\alpha"]\arrow[rd,hook,"{[1,2]}"]\arrow[dd,"{[1,2]}"swap]&&N\cP\\ 
&\Delta[0]\amalg_{\Delta[1]}\Delta[3]\amalg_{\Delta[1]}\Delta[0]\arrow[ru,dotted,"\overline\alpha"]&\\
\Delta[3]\arrow[ru]\arrow[rruu,dashed,"\widetilde\alpha",bend right,swap]
\end{tikzcd}
\]
Here, the dotted arrow $\overline\alpha$ exists by \cref{NPcontractibleKan}.
In particular, we obtained a map of simplicial sets
\[\widetilde\alpha\colon\Delta[3]\to N\cP,\]
which fits into the following commutative diagram of simplicial sets
\[
\begin{tikzcd}
\Delta[1]\arrow[d,"\alpha"swap]\arrow[r,"{[1,2]}"]&\Delta[3]\arrow[d,"\widetilde\alpha"]\\
N\cP\arrow[r,equal]&N\cP.    
\end{tikzcd}
\]
By construction, the map $\widetilde\alpha$ defines maps of marked simplicial sets
\[
\widetilde\alpha\colon\Delta[3]^{\mathrm{eq}}\to \mathrm{th}_1N^{\mathrm{RS}}\cP\quad\text{ and }\quad\widetilde\alpha\colon\Delta[3]^{\sharp}\to N^{\sharp}\cP,
\]
which fit into commutative diagram
commutative diagram
\[
\begin{tikzcd}
\Delta[1]\arrow[d,"{[1,2]}"swap]\arrow[r]\arrow[dd,bend right=70,"\alpha" swap]&\Delta[1]_t\arrow[d,"{[1,2]}"]\arrow[dd,bend left=70,"\alpha"]\\
\Delta[3]^{\mathrm{eq}}\arrow[r]\arrow[d,"\widetilde\alpha"swap]&\Delta[3]^\sharp\arrow[d,"\widetilde\alpha"]\\
\mathrm{th}_1N^{\mathrm{RS}}\cP\arrow[r]&\mathrm{th}_0N\cP.
\end{tikzcd}
\]
Indeed, the simplices $\widetilde\alpha_{[0,2]}$ and $\widetilde\alpha_{[1,3]}$ are degenerate by construction, and $\widetilde\alpha$, $\widetilde\alpha_{[0,1,3]}$, $\widetilde\alpha_{[1,2,3]}$, $\widetilde\alpha_{[0,1,2]}$ and $\widetilde\alpha_{[0,2,3]}$ have dimension higher than $1$.
Moreover, the lower square can be used to build a pushout of marked simplicial sets
\[
\begin{tikzcd}
\coprod_{\alpha}\Delta[3]^{\mathrm{eq}}\arrow[r]\arrow[d,"\coprod_{\alpha}\widetilde\alpha"swap]&\coprod_{\alpha}\Delta[3]^\sharp\arrow[d,"\coprod_{\alpha}\widetilde\alpha"]\\
\mathrm{th}_1N^{\mathrm{RS}}\cP\arrow[r]&N^\sharp\cP
\end{tikzcd}
\]
indexed over all $1$-simplices $\alpha$ of $N\cP$.
Since the top map is an acyclic cofibration in $m\sset_{\mathrm{cmp,sat}}$, then so is the bottom one, as desired.
\end{proof}

\begin{thm}
\label{theorem}
For every contractible
$\omega$-category $\cP$ and for all $\ell\geq0$ the canonical map
\[
\mathrm{th}_{\ell+1}N^{\mathrm{RS}}\cP\hookrightarrow \mathrm{th}_{\ell}N^{\mathrm{RS}}\cP
\]
is a weak equivalence in
$m\sset_{\mathrm{cmp,sat}}$.
\end{thm}

The proof will rely on a couple of auxiliary facts.

\begin{prop}
\label{RVmovingLemma}
Let $0\leq t\leq d-2$.
Let $A$ be a complicial set, and $\tau$ a $d$-simplex of $A$ whose restriction to the first copy of $\Delta[t]$ is degenerate at a vertex, which is represented by a map $\tau\colon\Delta[d]/\Delta[t]\to A$. Then there exists a $(d+1)$-simplex in $A$, represented by a map \[\widetilde\tau\colon\Delta[d+1]\to A\]
such that:
\begin{enumerate}
    \item $\widetilde\tau$ defines a map of marked simplicial sets
\[
\widetilde\tau\colon\Delta^{t+1}[d+1]\amalg_{\Delta[d+1]}\Delta^{t+2}[d+1]\to A;
\]
    \item $\tau$ is the $(t+1)$-st face of $\widetilde\tau$; that is, the diagram of simplicial sets commutes
\[
\begin{tikzcd}
\Delta[d]\arrow[d,"\tau"swap]\arrow[r,"d^{t+1}"]&\Delta[d+1]\arrow[d,"\widetilde\tau"]\\
A\arrow[r,equal]&A; 
\end{tikzcd}
\]
\item
  the restriction of the $(t+2)$-nd face of $\widetilde\tau$ to the first copy of $\Delta[t+1]$ is degenerate at a vertex; that is, it is represented by a map of simplicial sets
\[
d^{t+2}\widetilde\tau\colon\Delta[d]/\Delta[t+1]\to A.
\]
\end{enumerate}
In particular, there is a commutative diagram of marked simplicial sets
\[
\begin{tikzcd}
\Delta[d]\arrow[r,"d^{t+1}"]\arrow[d]\arrow[rd,"\tau"swap]&\Delta^{t+1}[d+1]\amalg_{\Delta[d+1]}\Delta^{t+2}[d+1]\arrow[d,"\widetilde\tau"]&\arrow[l,"d^{t+2}"swap]\arrow[d]\Delta[d]\arrow[ld,"d^{t+2}\widetilde\tau"]\\
\Delta[d]/\Delta[t]\arrow[r]&A&\Delta[d]/\Delta[t+1]\arrow[l]
\end{tikzcd}
\]
\end{prop}

\begin{proof}
Using the terminology of \cite[Definition~D.7.1]{RiehlVerityBookMoreElements}, the claim we are trying to prove can be rephrased as follows:
if $\tau$ is a $d$-simplex of $A$ which is fully degenerate on an initial simplex 
$\Delta[t]\subseteq\Delta[d]$, then there is a $(d+1)$-simplex $\widetilde{\tau}$ of $A$ that exhibits $\tau'=d^{t+2}\widetilde{\tau}$ as a \emph{complicial companion} of $\tau=d^{t+1}\widetilde{\tau}$ (specifically with the index $k=t+1$) and that moreover $\tau'$ is degenerate on the incremented initial simplex $\Delta[t+1]\subseteq \Delta[d]$; the only caveat is that $\tau'$ cannot become fully degenerate this way, which means that we must have $t+1<d$. We explain how this claim is implicit in the proof of \cite[Lemma~D.7.3]{RiehlVerityBookMoreElements} even though their statement as written only records a much weaker conclusion.

In the proof of \cite[Lemma~D.7.3]{RiehlVerityBookMoreElements}, Riehl--Verity prove that for each $d$-simplex $\alpha$ of $A$ there is a sequence of complicial companionships $\alpha=\alpha^1\simeq_{A_d} \alpha^2\simeq_{A_d} \dots \simeq_{A_d} \alpha^d$ where each $\alpha^j$ is fully degenerate on the initial simplex $\Delta[j-1]\subseteq\Delta[d]$ (starting with $\alpha$ which is fully degenerate on $\Delta[0]$, vacuously). Unraveling the inductive step, we see that the construction of $\alpha^{j+1}$ only depends on the immediately preceding simplex $\alpha^{j}$ (and not, for example, on $\alpha$ or the other $\alpha^i$ inbeetween) and has the property that the initial simplex $\alpha^{j+1}|_{\Delta[j]}$ is a degeneration of the initial simplex $\alpha^j|_{\Delta[j-1]}$ so that full degeneration of the latter directly implies full degeneration of the former. Renaming $j=t+1$, $\tau=\alpha^{j}$ and $\tau'=\alpha^{j+1}$, this is precisely the claim that we set out to establish.

To summarize, the claimed construction of $\widetilde\tau$ from $\tau$ is exactly the step in the proof of \cite[Lemma~D.7.3]{RiehlVerityBookMoreElements}, where (with their notation) $\tau^{t+2}$ is constructed from $\alpha^{t+1}$.
\end{proof}

We will use two instances of naturality for the map from \cref{ProgressiveSuspensionsGeneral}:
\begin{lem}
\label{lemmaSuspension1}
Let $\ell\geq0$ and $0\leq p \leq \ell+2$.
Let $S\subseteq \{0, \ldots, \ell+2\}$ such that $|S|\geq p+1$, and
\[\overline{S}\coloneqq S\setminus \{\ell+2-p+1,\ell+2-p+2, \ldots,\ell+2-p+p\}.\]
Denote by
\[\rho\colon \Delta[S]\to \Delta[\ell+2]\quad\text{ and }\quad \overline\rho\colon \Delta[\overline{S}]\to \Delta[\ell+2-p]\]
the corresponding maps.
If $\rho$ defines a non-degenerate $(|S|-1)$-simplex in $\Sigma^p\Delta[\ell+2-p]$,
then $\overline\rho$ defines a non-degenerate $(|\overline{S}|-1)$-simplex in $\Delta[\ell+2-p]$
and there is a commutative square of marked simplicial sets
\[
\begin{tikzcd}
\Delta[S]\arrow[r,two heads]\arrow[d,"\rho"swap]&
\Sigma^p\Delta[\overline{S}]
\arrow[d,"\Sigma^p\overline\rho"]\\
\Delta[\ell+2]
\arrow[r,two heads
]&\Sigma^p\Delta[\ell+2-p]
\end{tikzcd}
\]
\end{lem}

\begin{proof}
Let $X\coloneqq\Delta[\ell+2-p]$ and consider the classification of \cref{SurvivingNonDegenerate}.
Since $\dim \rho =|S|-1 \geq p$, the only case that can happen is $\rho=\sigma\star \tau$, where $\tau=\{\ell+2-p+1,\dots,\ell+2-p+p\}$ (corresponding to $\{0,\dots, p-1\}$ in the notation of \cref{SurvivingNonDegenerate}). In other words, $\rho$ is of the form
\[\sigma\star \Delta[p-1]\colon Y\star \Delta[p-1] \to X\star\Delta[p-1],\]
where $Y\coloneq \Delta[\overline{S}]$.
Therefore the claimed commutative square is an instance of the naturality of the quotient map from \cref{ProgressiveSuspensionsGeneral} applied to the map $\overline{\rho}\coloneq\sigma\colon Y\to X$.
\end{proof}

\begin{lem}
\label{lemmaSuspension2}
For $0\leq p\leq s$, there is a commutative square of marked simplicial sets
\[
\begin{tikzcd}
\Delta[s]\arrow[r]\arrow[d]&\Sigma^p\Delta[s-p]\arrow[d]\\
\Delta[s]_t\arrow[r]&\Sigma^p\Delta[s-p]_t
\end{tikzcd}
\]
\end{lem}

\begin{proof}
This is an instance of the naturality of the quotient map from \cref{ProgressiveSuspensionsGeneral} applied to the map $\Delta[s]\to\Delta[s]_t$.
\end{proof}

We will also use \cref{ProgressiveSuspensionsGeneral} to produce relevant filtrations of map $\Sigma^p\Delta[\ell+1]\to\Sigma^{\ell+1}\Delta[0]$:

\begin{rmk}
\label{ProgressiveSuspensions}
Specializing the construction from \cref{ProgressiveSuspensionsGeneral2} to the case $X=\Delta[0]$ we obtain the filtration
\[
\Delta[\ell+1]\to\Sigma\Delta[\ell]\to\dots\to\Sigma^p\Delta[\ell-p+1]\to\Sigma^{p+1}\Delta[\ell-p]\to\dots\to\Sigma^{\ell}\Delta[1]\to \Sigma^{\ell+1}\Delta[0],
\]
with $0\leq p\leq\ell$. Further, following the convention that each quotient collapses the first copy of a simplex,
each of intermediate quotient maps can be factored into intermediate quotient maps
\[
\Sigma^p\Delta[\ell-p+1]=\Sigma^p(\Delta[\ell-p+1]/\Delta[0])\to\Sigma^p(\Delta[\ell-p+1]/\Delta[1])\to\dots\quad\quad\]
\[\dots\to \Sigma^p(\Delta[\ell-p+1]/\Delta[t])\to \Sigma^p(\Delta[\ell-p+1]/\Delta[t+1])\to\dots\]
\[\dots\to\Sigma^p(\Delta[\ell-p+1]/\Delta[\ell-p])=\Sigma^p(\Sigma\Delta[\ell-p])=\Sigma^{p+1}\Delta[\ell-p]
\]
with  $0\leq t\leq \ell-p-1$.
\end{rmk}

\begin{lem}
\label{lemmaV2}
Let $\ell\geq 0$ and $0\leq p\leq s\leq  \ell +2$.
Fix a strict $\omega$-category $\cP$,
an injective map
\[\overline{\rho}\colon \Delta[s-p]\to \Delta[\ell+2-p],\]
and a map of marked simplicial sets.
\[\widetilde{\gamma}\colon\Sigma^p\Delta[\ell+2-p]\to N^{RS}\cP.\]
Assume that the $(s-p)$-simplex
\[\Delta[s-p]\xrightarrow{\overline{\rho}} \Delta[\ell+2-p]\xrightarrow{\widetilde{\beta}} N^{\mathrm{RS}}(\cP^{(p)}(\glboundary{\widetilde{\gamma}}))\]
is marked,
where $\widetilde{\beta}$ is the transpose of $\widetilde{\gamma}$
under the adjunction \cref{IteratedSuspensionAdjunctionMarked}
(using the identification of \cref{NvsIteratedHom} and the notation from \cref{GlobularBoundary}).
Then the composite 
\[
\Delta[s]\twoheadrightarrow\Sigma^p\Delta[s-p]\xrightarrow{\Sigma^p\overline\rho}\Sigma^p\Delta[\ell+2-p]\xrightarrow{\widetilde\gamma}N^{\mathrm{RS}}\cP
\]
is a marked $s$-simplex of $N^{\mathrm{RS}}\cP$.
\end{lem}

\begin{proof}
We have to show that there is a solution to the lifting problem in $m\sset$
\[
\begin{tikzcd}
\Delta[s]
\arrow[rd,two heads]
\arrow[dd]
&
&&
N^{\mathrm{RS}}\cP
\\
&
\Sigma^p\Delta[s-p]
\arrow[r,"\Sigma^p\overline\rho" swap]
&
\Sigma^p\Delta[\ell+2-p]
\arrow[ru,"\widetilde\gamma"swap]
&
\\
\Delta[s]_t
\arrow[rrruu,dashed,bend right=50,"?" swap]
&
&
&
\\
&
&&
\end{tikzcd}
\]
By \cref{lemmaSuspension2}, it suffices to solve the lifting problem in $m\sset$
\[
\begin{tikzcd}
\Delta[s]
\arrow[rd,two heads]
\arrow[dd]
&
&&
N^{\mathrm{RS}}\cP
\\
&
\Sigma^p\Delta[s-p]
\arrow[r,"\Sigma^p\overline\rho" swap]
\arrow[dd]
&
\Sigma^p\Delta[\ell+2-p]
\arrow[ru,"\widetilde\gamma"swap]
&
\\
\Delta[s]_t
\arrow[rd,two heads]
&
&
&
\\
&
\Sigma^p\Delta[s-p]_t
\arrow[rruuu,dashed,bend right,"?" swap]
&&
\end{tikzcd}
\]
Hence, it suffices to solve the lifting problem in $m\sset$
\[
\begin{tikzcd}
&&
N^{\mathrm{RS}}\cP
\\
\Sigma^p\Delta[s-p]
\arrow[r,"\Sigma^p\overline\rho"]
\arrow[d]
&
\Sigma^p\Delta[\ell+2-p]
\arrow[ru,"\widetilde\gamma"swap]
&
\\
\Sigma^p\Delta[s-p]_t
\arrow[rruu,dashed,bend right,"?" swap]
&&
\end{tikzcd}
\]
Observe that, by \cref{GlobularBoundary}, the lifting problem actually lives in ${}^{\Sigma^{p}\varnothing/}m\sset$.
By \cref{IteratedSuspensionAdjunctionMarked}, this is equivalent to solving
the transposed lifting problem in $m\sset$
\[
\begin{tikzcd}
&&
N^{\mathrm{RS}}\bigl(\cP^{(p)}(\glboundary\widetilde{\gamma})\bigr)
\\
\Delta[s-p]
\arrow[r,"\overline\rho"]
\arrow[d]
&
\Delta[\ell+2-p]
\arrow[ru,"\widetilde\beta"]
&
\\
\Delta[s-p]_t
\arrow[rruu,dashed,bend right,"?" swap]
&&
\end{tikzcd}
\]
Such a lift exists precisely because
$\widetilde\beta\circ\overline\rho$ is marked by assumption.
Hence $\widetilde\gamma\circ\Sigma^p\overline\rho$ is marked, as desired.
\end{proof}

We can now prove \cref{theorem}.

\begin{proof}[Proof of \cref{theorem}]
The statement
is proven by induction on $\ell\geq0$. We treated the base case $\ell=0$ in \cref{InductionBase} and we treat the inductive step for general $\ell$ now. 

Given $\sigma\colon\Delta[\ell+1]\to N\cP$ an $(\ell+1)$-simplex of $N\cP$,
we say that $\sigma$ has
\emph{depth} $p$ for $0\leq p\leq \ell$ and \emph{type} $t$
for $0\leq t\leq\ell-p$ if it factors through 
\[
\sigma\colon \Sigma^p(\Delta[\ell+1-p]/\Delta[t])\to N\cP,
\]
referring to the maps from \cref{ProgressiveSuspensions},
and $(p,t)$ is maximal with this property (in lexicographic order).
For $0\leq p\leq \ell$ and $0\leq t\leq\ell-p$
let $W_{\cP,\ell+1}^{(p,t)}$ denote the marked simplicial set obtained from
$\mathrm{th}_{\ell+1}N^{RS}\cP$ by marking all $(\ell+1)$-simplices
of depth and type at least $(p,t)$.

Observe that for $0\leq p<\ell$ and $t=\ell-p$ we have
\[
  \Sigma^p(\Delta[\ell+1-p]/\Delta[t])=\Sigma^p(\Sigma\Delta[\ell-p])
  =\Sigma^{p+1}(\Delta[\ell-p]/\Delta[0]);
\]
hence if a simplex has depth and type at least $(p,\ell-p)$,
then it already has depth and type at least $(p+1,0)$.
In other words, the inclusion
\[
  W_{\cP,\ell+1}^{(p,\ell-p)}\xrightarrow{=}W_{\cP,\ell+1}^{(p+1,0)}
\]
is an equality.

The desired inclusion factors as follows:
\[
\mathrm{th}_{\ell+1}N^{RS}\cP\hookrightarrow W_{\cP,\ell+1}^{(\ell,0)}\hookrightarrow W_{\cP,\ell+1}^{(\ell-1,0)}\hookrightarrow\dots
\hookrightarrow W_{\cP,\ell+1}^{(p+1,0)}\hookrightarrow W_{\cP,\ell+1}^{(p,0)}\hookrightarrow\dots\quad\quad\quad\quad\quad\quad\]
\[\quad\quad\quad\quad\quad\quad\quad\quad\quad\quad\quad\quad\quad\quad\quad\quad\quad \dots\hookrightarrow W_{\cP,\ell+1}^{(1,0)} \hookrightarrow W_{\cP,\ell+1}^{(0,0)} =\mathrm{th}_{\ell}N^{RS}\cP.
\]

We first show that
\[\mathrm{th}_{\ell+1}N^{RS}\cP\hookrightarrow W_{\cP,\ell+1}^{(\ell,0)}\]
is an acyclic cofibration in $m\sset_{\mathrm{cmp,sat}}$.
Given any $(\ell+1)$-simplex $\sigma$ of $N\cP$ of depth $\ell$ we construct a map of simplicial sets $\widetilde\sigma\colon\Sigma^\ell\Delta[3]\to N\cP$
as follows.
An $(\ell+1)$-simplex $\sigma$ in $N\cP$ of depth $\ell$ is represented by a map of simplicial sets
\[
\sigma\colon\Sigma^\ell\Delta[1]\to N\cP.
\]
By \cref{GlobularBoundary}, this is also a map of simplicial sets under $\Sigma^{\ell}\varnothing$
\[
\begin{tikzcd}
&\Sigma^{\ell}\varnothing\arrow[ld,""swap]\arrow[rd,"\glboundary\sigma"]&\\
\Sigma^\ell\Delta[1]\arrow[rr,"\sigma"swap]&&N\cP.
\end{tikzcd}
\]
Using \cref{CvsIteratedSigma}, this corresponds through the adjunction from \cref{IteratedSuspensionAdjunctionMarked} to a map of simplicial sets
\[
\alpha\colon\Delta[1]\xrightarrow{} N(\cP^{(\ell)}(\glboundary\sigma)).
\]
By \cref{E:hom}, the $\omega$-category $\cP^{(\ell)}(\glboundary\sigma)$ is a contractible $\omega$-category.
Hence, we can perform on its $1$-simplex $\alpha$ the construction described in the proof of \cref{InductionBase},
which yields the existence of a $3$-simplex
\[
\widetilde\alpha\colon\Delta[3]\to N(\cP^{(\ell)}(\glboundary\sigma)).
\]
which fits into the following diagram of simplicial sets
\[
\begin{tikzcd}
\Delta[1]\arrow[d,"\alpha"swap]\arrow[r,"{[1,2]}"]&\Delta[3]\arrow[d,"\widetilde\alpha"]\\
N(\cP^{(\ell)}(\glboundary\sigma))\arrow[r,equal]&N(\cP^{(\ell)}(\glboundary\sigma))
\end{tikzcd}
\]
Moreover, by construction, the map $\widetilde\alpha$ of simplicial sets defines maps of marked simplicial sets
\[
\widetilde{\alpha}\colon\Delta[3]^{\mathrm{eq}}\to \mathrm{th}_{1}N^{RS}(\cP^{(\ell)}(\glboundary\sigma))\quad\text{ and }\quad\widetilde{\alpha}\colon\Delta[3]^{\sharp}\to \mathrm{th}_{0}N^{RS}(\cP^{(\ell)}(\glboundary\sigma)),
\]
which fit into the following commutative diagram of marked simplicial sets
\[
\begin{tikzcd}
\Delta[1]\arrow[d,"{[1,2]}"swap]\arrow[r]\arrow[dd,bend right=70,"\alpha"swap]&\Delta[1]_t\arrow[d,"{[1,2]}"]\arrow[dd,bend left=70,"\alpha"]\\
\Delta[3]^{\mathrm{eq}}\arrow[r]\arrow[d,"\widetilde\alpha"swap]&\Delta[3]^\sharp\arrow[d,"\widetilde\alpha"]\\
\mathrm{th}_1N^{RS}(\cP^{(\ell)}(\glboundary\sigma))\arrow[r]&\mathrm{th}_0N(\cP^{(\ell)}(\glboundary\sigma))
\end{tikzcd}
\]

The map of simplicial sets $\widetilde{\alpha}$ corresponds through the adjunction from \cref{IteratedSuspensionAdjunctionMarked} to a map of simplicial sets under $\Sigma^{\ell}\varnothing$
\[
\begin{tikzcd}
&\Sigma^{\ell}\varnothing\arrow[ld,""swap]\arrow[rd,"\glboundary\sigma"]&\\
\Sigma^\ell\Delta[3]\arrow[rr,"\widetilde\sigma"swap]&&N\cP.
\end{tikzcd}
\]
In particular, we obtained a map of simplicial sets 
\[
\widetilde\sigma\colon\Sigma^\ell\Delta[3]\to N\cP,
\]
with $\glboundary\sigma=\glboundary\widetilde\sigma$,
which fits into the following commutative diagram of simplicial sets
\[
\begin{tikzcd}
\Sigma^\ell\Delta[1]\arrow[d,"\sigma"swap]\arrow[rr,"{\Sigma^\ell[1,2    ]}"]&&\Sigma^\ell\Delta[3]\arrow[d,"\widetilde\sigma"]\\
N\cP\arrow[rr,equal]&&N\cP.
\end{tikzcd}
\]
We can show that the map $\widetilde\sigma$ of simplicial sets defines maps of marked simplicial sets
\[
\widetilde{\sigma}\colon\Sigma^\ell\Delta[3]^{\mathrm{eq}}\to \mathrm{th}_{\ell+1}N^{RS}\cP\quad\text{ and }\quad\widetilde{\sigma}\colon\Sigma^\ell\Delta[3]^{\sharp}\to W_{\cP,\ell+1}^{(\ell,0)},
\]
and it fits into the following commutative diagram of marked simplicial sets
\[
\begin{tikzcd}
\Sigma^\ell\Delta[1]\arrow[d,"{\Sigma^\ell[1,2]}"swap]\arrow[r]\arrow[dd,bend right=90,"\sigma" swap]&
\Sigma^\ell\Delta[1]_t\arrow[d,"{\Sigma^\ell[1,2]}"]\arrow[dd,bend left=80,"\sigma"]\\
\Sigma^\ell\Delta[3]^{\mathrm{eq}}\arrow[r]\arrow[d,"\widetilde\sigma"swap]&\Sigma^\ell\Delta[3]^\sharp\arrow[d,"\widetilde\sigma"]\\
\mathrm{th}_{\ell+1}N^{RS}\cP\arrow[r]&W^{(\ell,0)}_{\cP,\ell+1}.
\end{tikzcd}
\]
To see this, recall from \cref{IteratedSuspensionMarked} that all non-degenerate marked $k$-simplices of $\Sigma^\ell\Delta[3]^{\mathrm{eq}}$(resp.~$\Sigma^\ell\Delta[3]^\sharp$) are of the form $\Sigma^\ell\tau$, with $\tau$ marked non-degenerate $(k-\ell)$-simplex
in $\Delta[3]^{\mathrm{eq}}$ (resp.~$\Delta[3]^\sharp$).
We then observe the following:
\begin{itemize}[leftmargin=*]
    \item The simplices $\widetilde\sigma_{\Sigma^\ell[0,2]}$ and $\widetilde\sigma_{\Sigma^\ell[1,3]}$ are degenerate $(\ell+1)$-simplices in $N^{RS}\cP$, hence marked in $N^{RS}\cP$, because $\widetilde\alpha_{[0,2]}$ and $\widetilde\alpha_{[1,3]}$ are by construction degenerate $1$-simplices in $N\cP^{(\ell)}(\glboundary\sigma)$.
    \item For all $0\leq i< j< k\leq 3$, the simplices $\widetilde\sigma_{\Sigma^\ell[0,1,2,3]}$ and $\widetilde\sigma_{\Sigma^\ell[i,j,k]}$
    have dimension higher than $\ell+1$ in $N\cP$, because $\widetilde\alpha_{[0,1,2,3]}$ and $\widetilde\alpha_{[i,j,k]}$ have dimension higher than $1$ in $\cP(\ell)(\glboundary\sigma)$.
    \item The simplices $\widetilde\sigma\circ{\Sigma^\ell[0,1]}$, $\widetilde\sigma\circ{\Sigma^\ell[1,2]}$,  $\widetilde\sigma\circ{\Sigma^\ell[2,3]}$, and $\widetilde\sigma\circ{\Sigma^\ell[0,3]}$ have depth at least $\ell$ in $N\cP$ by construction.
\end{itemize}
Hence, all non-degenerate marked simplices of $\Sigma^\ell\Delta[3]^{\mathrm{eq}}$ (resp.~$\Sigma^\ell\Delta[3]^\sharp$) are sent to marked simplices of $\mathrm{th}_{\ell+1}N^{RS}\cP$ (resp.~$W^{(\ell,0)}_{\cP,\ell+1}$), validating the existence of the desired diagram of marked simplicial sets.

Moreover, the lower square can be used to build a pushout of marked simplicial sets
\[
\begin{tikzcd}
\coprod_{\sigma}\Sigma^\ell\Delta[3]^{\mathrm{eq}}\arrow[r]\arrow[d,"\coprod_{\sigma}\widetilde\sigma"swap]&\coprod_{\sigma}\Sigma^\ell\Delta[3]^\sharp\arrow[d,"\coprod_{\sigma}\widetilde\sigma"]\\
\mathrm{th}_{\ell+1}N^{RS}\cP\arrow[r]&W^{(\ell,0)}_{\cP,\ell+1}
\end{tikzcd}
\]
indexed over all $(\ell+1)$-simplices $\sigma$ of $N\cP$ of depth $\ell$.
Since the top map is an acyclic cofibration in $m\sset_{\mathrm{cmp,sat}}$, then so is the bottom one, as desired.

Next, we show that
\[
W_{\cP,\ell+1}^{(p+1,0)}\hookrightarrow W_{\cP,\ell+1}^{(p,0)}
\]
is an acyclic cofibration in $m\sset_{\mathrm{cmp,sat}}$ for $\ell>p$.
We factor the desired inclusion as
\[
W_{\cP,\ell+1}^{(p+1,0)}= W_{\cP,\ell+1}^{(p,\ell-p)}\hookrightarrow W_{\cP,\ell+1}^{(p,\ell-p-1)}\hookrightarrow\dots
\hookrightarrow W_{\cP,\ell+1}^{(p,t+1)}\hookrightarrow W_{\cP,\ell+1}^{(p,t)}\hookrightarrow \dots\hookrightarrow W_{\cP,\ell+1}^{(p,1)}\hookrightarrow W_{\cP,\ell+1}^{(p,0)}.
\]
For this, we show that
\[
W_{\cP,\ell+1}^{(p,t+1)}\to W_{\cP,\ell+1}^{(p,t)}
\]
is an acyclic cofibration in $m\sset_{\mathrm{cmp,sat}}$ for all $t=\ell-p-1,\dots,0$.

Given any $(\ell+1)$-simplex $\gamma$ of $N\cP$ of depth $p$ and type $t$ we construct a map
\[
\widetilde\gamma\colon\Sigma^p\Delta[\ell+2-p]\to N\cP.
\]
Given
\[\gamma\colon\Sigma^p\Delta[\ell+1-p]\to N\cP,\]
using the adjunctions from 
\cref{IteratedSuspensionAdjunctionCat,IteratedSuspensionAdjunctionMarked,NerveAdjunctionMarked}
similarly to the previous argument,
we obtain a map of simplicial sets of type $t$
\[
\beta\colon\Delta[\ell+1-p]\to N(\cP^{(p)}(\glboundary\gamma)).
\]
By \cite[Theorem 249]{VerityComplicialAMS},
the marked simplicial set $N^{\mathrm{RS}}(\cP^{(p)}(\glboundary\gamma))$ is a complicial set. Since $0\leq t\leq\ell-p-1$,
we can perform on its $(\ell+1-p)$-simplex $\beta$ the construction from \cref{RVmovingLemma}, which yields an $(\ell+2-p)$-simplex of $N(\cP^{(p)}(\glboundary\gamma))$, so a map of simplicial sets
\[
\widetilde\beta\colon\Delta[\ell+2-p]\to N(\cP^{(p)}(\glboundary\gamma)),
\]
which fits into the following diagram of simplicial sets
\[
\begin{tikzcd}
\Delta[\ell+1-p]\arrow[d,"\beta"swap]\arrow[r,"d^{t+1}"]&\Delta[\ell+2-p]\arrow[d,"\widetilde\beta"]\\
N(\cP^{(p)}(\glboundary\gamma))\arrow[r,equal]&N(\cP^{(p)}(\glboundary\gamma)).
\end{tikzcd}
\]
We claim that it fits into the following commutative diagram of marked simplicial sets
\[
\begin{tikzcd}
\Delta[\ell+1-p]\arrow[d,"d^{t+1}"swap]\arrow[r]\arrow[dd,bend right=70,"\beta"swap]&\Delta[\ell+1-p]_t\arrow[d,"d^{t+1}"]\arrow[dd,bend left=70,"\beta"]\\
\Delta^{t+1}[\ell+2-p]'\arrow[r]\arrow[d,"\widetilde\beta"swap]&\Delta^{t+1}[\ell+2-p]''\arrow[d,"\widetilde\beta"]\\
W_{\cP^{(p)}(\glboundary\gamma),\ell+1-p}^{(0,t+1)}\arrow[r]& \quad W_{\cP^{(p)}(\glboundary\gamma),\ell+1-p}^{(0,t)}
\end{tikzcd}
\]
To see this, we observe that:
\begin{itemize}[leftmargin=*]
    \item The simplex $\widetilde\beta$ defines a map of marked simplicial sets $\Delta^{t+1}[\ell+2-p]\to N^{RS}(\cP^{(p)}(\glboundary\gamma))$.
    \item The $t$-th face $d_t(\widetilde\beta)$ of $\widetilde\beta$ is marked in $N^{RS}\cP^{(p)}(\glboundary\gamma)$, because the simplex $\widetilde\beta$ defines by construction a
    map of marked simplicial sets $\Delta^{t+2}[\ell+2-p]\to N^{RS}(\cP^{(p)}(\glboundary\gamma))$.
  \item
    The $(t+2)$-nd face $d_{t+2}(\widetilde\beta)$ of $\widetilde\beta$
    is by construction at least of depth $0$ and type $t+1$ in $N^{RS}(\cP^{(p)}(\glboundary\gamma))$.
  \item
    The $(t+1)$-st face of $\widetilde\beta$
    is $d_{t+1}(\widetilde\beta)=\beta$,
    which is of type $t$ in $N^{RS}(\cP^{(p)}(\glboundary\gamma))$.
\end{itemize}
Hence, all non-degenerate marked simplices of $\Delta^{t+1}[\ell+2-p]'$ (resp.~$\Delta^{t+1}[\ell+2-p]''$)
are sent to marked simplices of $W_{\cP^{(p)}(\glboundary\gamma),\ell+1-p}^{(0,t+1)}$ (resp.~$W_{\cP^{(p)}(\glboundary\gamma),\ell+1-p}^{(0,t)}$), validating the existence of the desired diagram of marked simplicial sets.

In particular, using the adjunctions from \cref{IteratedSuspensionQuillenCat} similarly to the previous argument,
we obtain a map of simplicial sets
\[
\widetilde\gamma\colon\Sigma^p\Delta[\ell+2-p]\to N\cP,
\]
which fits into the following commutative diagram of simplicial sets
\[
\begin{tikzcd}
\Sigma^p\Delta[\ell+1-p]\arrow[d,"\gamma"swap]\arrow[r,"\Sigma^p d^{t+1}"]&\Sigma^p\Delta[\ell+2-p]\arrow[d,"\widetilde\gamma"]\\
N\cP\arrow[r,equal]&N\cP.
\end{tikzcd}
\]
We claim that the map $\widetilde\gamma$ defines maps of marked simplicial sets
\[
\widetilde\gamma\colon\Sigma^p\Delta^{t+1}[\ell+2-p]'\to W_{\cP,\ell+1}^{(p,t+1)}
\quad\text{ and }\quad \widetilde\gamma\colon\Sigma^p\Delta^{t+1}[\ell+2-p]''\to W_{\cP,\ell+1}^{(p,t)},\]
which fit into the following commutative diagram of marked simplicial sets
\[
\begin{tikzcd}
\Sigma^p\Delta[\ell+1-p]\arrow[d,"\Sigma^pd^{t+1}"swap]\arrow[r,""]\arrow[dd,bend right=90,"\gamma"swap]&\Sigma^p\Delta[\ell+1-p]_t\arrow[d,"\Sigma^pd^{t+1}"]\arrow[dd,bend left=90,"\gamma"]\\
\Sigma^p\Delta^{t+1}[\ell+2-p]'\arrow[r]\arrow[d,"\widetilde\gamma"swap]&\Sigma^p\Delta^{t+1}[\ell+2-p]''\arrow[d,"\widetilde\gamma"]\\
W_{\cP,\ell+1}^{(p,t+1)} \arrow[r]&W_{\cP,\ell+1}^{(p,t)}
\end{tikzcd}
\]
To see this, recall from \cref{IteratedSuspensionMarked} that all non-degenerate marked $k$-simplices of $\Sigma^p\Delta[\ell+1-p]'$ (resp.~$\Sigma^p\Delta[\ell+1-p]''$) are of the form $\Sigma^p\tau$,
with $\tau$ marked non-degenerate $(k-p)$-simplex
in $\Delta[\ell+1-p]'$ (resp.~$\Delta[\ell+1-p]''$).
We then observe the following:
\begin{itemize}[leftmargin=*]
\item The simplex $\widetilde\gamma$ defines a
map of marked simplicial sets $\Delta^{t+1}[\ell+2]\to N^{RS}\cP$,
which follows from the fact that
the simplex $\widetilde\beta$ defines map of marked simplicial sets $\Delta^{t+1}[\ell+2-p]\to N^{RS}(\cP^{(p)}(\glboundary\gamma))$.

To see this, we show that, for all
$\{t,t+1,t+2\}\subseteq S\subseteq\{0,\dots,\ell+2\}$, writing
$s=|S|-1$ and letting
\[
\rho\colon\Delta[s]\to\Delta[\ell+2]
\]
be the induced map, the simplex
\[
\widetilde\gamma\circ\rho\colon
\Delta[s]\to N^{\mathrm{RS}}\cP
\]
is marked in $N^{\mathrm{RS}}\cP$.
Let $\{t,t+1,t+2\}\subseteq S\subseteq\{0,\dots,\ell+2\}$
and let
$\rho\colon\Delta[s]\to\Delta[\ell+2]$ be the corresponding map. Given that
$0\leq t\leq\ell-p-1$, we have
$t+2\leq\ell-p+1<\ell+2-p$,
so we have
$\{t,t+1,t+2\}\subseteq S\cap\{0,\ldots,\ell+2-p\}$,
so $S\cap\{0,\ldots,\ell+2-p\}$ has at least $3$ elements.
\begin{itemize}
\item If $\widetilde\gamma\circ\rho$ is a degenerate $s$-simplex in $N^{\mathrm{RS}}\cP$, then it defines a marked $s$-simplex in $N^{\mathrm{RS}}\cP$.
\item If $\widetilde\gamma\circ\rho$ is not a degenerate $s$-simplex in $N^{\mathrm{RS}}\cP$, by (the contrapositive of) \cref{SurvivingNonDegenerate}, we have $s\geq p$, so \cref{lemmaSuspension1} applies. By \cref{lemmaSuspension1}, we obtain that there is a commutative diagram
\[
\begin{tikzcd}
\Delta[s]
\arrow[r,two heads]
\arrow[d,"\rho"swap]
&
\Sigma^p\Delta[s-p]
\arrow[d,"\Sigma^p\overline\rho"]
\\
\Delta[\ell+2]
\arrow[r,two heads]
&
\Sigma^p\Delta[\ell+2-p].
\end{tikzcd}
\]
for an injective map
\[\overline\rho\colon\Delta[s-p]\to\Delta[\ell+2-p].\]
We record this, with a slight abuse of notation, as $\rho=\Sigma^p\overline\rho$. In particular, we obtain that
$\widetilde\gamma\circ\rho=\widetilde\gamma\circ\Sigma^p\overline\rho$. Moreover, since $\{t,t+1,t+2\}\subseteq S\cap\{0,\ldots,\ell+2-p\}$, the image of $\overline\rho$ contains $\{t,t+1,t+2\}$. Thus $\overline\rho$ defines a marked $(s-p)$-simplex of $\Delta^{t+1}[\ell+2-p]$. Since the simplex $\widetilde\beta$ defines a map of marked simplicial sets $\Delta^{t+1}[\ell+2-p]\to N^{\mathrm{RS}}(\cP^{(p)}(\glboundary\gamma))$, it follows that $\widetilde\beta\circ\overline\rho$ defines a marked $(s-p)$-simplex in $N^{\mathrm{RS}}(\cP^{(p)}(\glboundary\gamma))$. By \cref{lemmaV2}, we obtain that $\widetilde\gamma\circ\rho=\widetilde\gamma\circ\Sigma^p\overline\rho$ defines a marked $s$-simplex in $N^{\mathrm{RS}}\cP$, as desired.
\end{itemize}

\item The $t$-th face $d_t(\widetilde\gamma)
$ of $\widetilde\gamma$ is marked in $N^{RS}\cP$, which follows from the fact that the $t$-th face $d_t(\widetilde\beta)$ is marked in $N^{RS}\cP^{(p)}(\glboundary\gamma)$.

To see this, consider the coface map
\[d^t\colon\Delta[\ell+1-p]\to\Delta[\ell+2-p].\]
Given that $t\leq \ell-p-1$, by \cref{lemmaSuspension1} there is a commutative diagram
\[
\begin{tikzcd}
\Delta[\{0,\ldots,\widehat{t},\ldots,\ell+2\}]
\arrow[r,two heads]
\arrow[d,"d^t"swap]
&
\Sigma^p\Delta[\{0,\ldots,\widehat{t},\ldots,\ell+2-p\}]
\arrow[d,"\Sigma^p d^t"]
\\
\Delta[\ell+2]
\arrow[r,two heads]
&
\Sigma^p\Delta[\ell+2-p].
\end{tikzcd}
\]
We record this, with a slight abuse of notation, as $d^t=\Sigma^p d^t$. Thus $d_t(\widetilde\gamma)=\widetilde\gamma\circ d^t=\widetilde\gamma\circ\Sigma^p d^t$. Since the $(\ell+1-p)$-simplex $d_t(\widetilde\beta)=\widetilde\beta\circ d^t$
is marked in $N^{RS}(\cP^{(p)}(\glboundary\gamma))$, by \cref{lemmaV2} we obtain that $d_t(\widetilde\gamma)=\widetilde\gamma\circ\Sigma^p d^t$ defines a marked $(\ell+1)$-simplex in $N^{RS}\cP$, as desired.

\item
  For $j\in\{t+1,t+2\}$, the $j$-th face $d_j(\widetilde\gamma)$ of $\widetilde\gamma$
  is (at least) of depth $p$ and type $j-1$ in $N^{RS}\cP$,
  because the $j$-th face $d_j(\widetilde\beta)$ of $\widetilde\beta$
  is (at least) of depth $0$ and type $j-1$ in $N^{RS}(\cP^{(p)}(\glboundary\gamma))$.

To see this, consider the coface map
\[d^j\colon\Delta[\ell+1-p]\to\Delta[\ell+2-p].\]
Given that $t\leq \ell-p-1$, by \cref{lemmaSuspension1} there is a commutative diagram
\[
\begin{tikzcd}
\Delta[\{0,\ldots,\widehat{j},\ldots,\ell+2\}]
\arrow[r,two heads]
\arrow[d,"d^j"swap]
&
\Sigma^p\Delta[\{0,\ldots,\widehat{j},\ldots,\ell+2-p\}]
\arrow[d,"\Sigma^p d^j"]
\\
\Delta[\ell+2]
\arrow[r,two heads]
&
\Sigma^p\Delta[\ell+2-p].
\end{tikzcd}
\]
We record this, with a slight abuse of notation, as $d^j=\Sigma^p d^j$.
Thus $d_j(\widetilde\gamma)=\widetilde\gamma\circ d^j=\widetilde\gamma\circ\Sigma^p d^j$.
Since the $(\ell+1-p)$-simplex $d_j(\widetilde\beta)=\widetilde\beta\circ d^j$
is (at least) of depth $0$ and type $j-1$ in $N^{RS}(\cP^{(p)}(\glboundary\gamma))$,
we obtain that $d_j(\widetilde\gamma)=\widetilde\gamma\circ\Sigma^p d^j$
defines an $(\ell+1)$-simplex which is (at least) of depth $p$ and type $j-1$ in $N^{RS}\cP$, as desired.
\end{itemize}
Hence, all non-degenerate marked simplices of $\Sigma^p\Delta^{t+1}[\ell+2-p]'$ (resp.~$\Sigma^p\Delta^{t+1}[\ell+2-p]''$) are sent to marked simplices of 
$W_{\cP,\ell+1}^{(p,t+1)}$ (resp.~$W_{\cP,\ell+1}^{(p,t)}$), validating the existence of the desired diagram of marked simplicial sets.

Moreover, the lower square can be used to build a pushout of marked simplicial sets
\[
\begin{tikzcd}
  \coprod_{\gamma}\Sigma^p\Delta^{t+1}[\ell+2-p]'\arrow[r]\arrow[d,"\coprod_{\gamma}\widetilde{\gamma}"swap]  &\coprod_{\gamma}\Sigma^p\Delta^{t+1}[\ell+2-p]''\arrow[d,"\coprod_{\gamma}\widetilde\gamma"]\\
W_{\cP,\ell+1}^{(p,t+1)} \arrow[r]&W_{\cP,\ell+1}^{(p,t)}
\end{tikzcd}
\]
indexed over all $(\ell+1)$-simplices $\gamma$ of $N\cP$ of depth $p$ and type $t$. Since the top map is an acyclic cofibration in $m\sset_{\mathrm{cmp,sat}}$ (by \cref{IteratedQuillenAdjunctionMarked}), then so is the bottom one, as desired.
\end{proof}

We can now conclude with the main results.

\begin{prop}
\label{NEacyclicfibrant}
For every contractible $\omega$-category $\cP$ the map of marked simplicial sets \[\mathrm{th}_0N^{\mathrm{RS}}\cP\to\Delta[0]\] is an acyclic fibration in $m\sset_{\mathrm{cmp,sat}}$.
\end{prop}

\begin{proof}
By \cref{NPcontractibleKan}, the map of simplicial sets $N\cP\to \Delta[0]$ is a weak equivalence in the Kan--Quillen model structure. By \cite[Lemma~1.27]{BOR},
the map of simplicial sets $\mathrm{th_0}N\cP\to \Delta[0]$ is a weak equivalence in the model structure $m\sset_{(\infty,0)}$. By \cref{TruncationsQuillenMarked}\ref{QPnleft}, the map of marked simplicial sets  $\mathrm{th_0}N\cP\to \Delta[0]$ is a weak equivalence in the model structure $m\sset_{\mathrm{cmp,sat}}$, concluding the proof.
\end{proof}

\begin{thm}
\label{LnEcontractible}
For every contractible $\omega$-category $\cP$ and for all $n\geq0$ the map of marked simplicial sets
\[\mathrm{th}_n N^{\mathrm{RS}}\cP\to\Delta[0]\]
is a weak equivalence in $m\sset_{\mathrm{cmp,sat}}$.
\end{thm}

\begin{proof}
This follows from the 2-out-of-3 property for weak equivalences in $m\sset_{\textrm{cmp,sat}}$ combined with \cref{theorem,NEacyclicfibrant}.
\end{proof}

Combining \cref{NEfibrant,NEnotContractible,LnEcontractible}, we established the following:
\begin{thm}
\label{CompleteTheorem}
If $\cE$ denotes the walking coinductive equivalence from \cite{ORsurvey,HLOR}, then the following hold:
\begin{enumerate}
  \setcounter{enumi}{-1}
    \item$N^{\mathrm{RS}}\cE$ is a saturated complicial set;
    \item $N^{\mathrm{RS}}\cE$ is not contractible in the homotopy theory of saturated complicial sets;
    \item for every $n\geq0$, the thinification $\mathrm{th}_nN^{\mathrm{RS}}\cE$ is contractible in the homotopy theory of saturated complicial sets.
\end{enumerate}
\end{thm}

\bibliographystyle{amsalpha}
\bibliography{Ref}

\end{document}